\documentclass[a4paper, 10pt, twoside, notitlepage]{amsart}

\usepackage[utf8]{inputenc}
\usepackage{color}
\usepackage{amsmath} 
\usepackage{amssymb} 
\usepackage{amsthm}
\usepackage{geometry}
\usepackage{graphicx}
\usepackage{esint}
\usepackage[colorlinks=true,linkcolor=blue]{hyperref}

\theoremstyle{plain}
\newtheorem{thm}{Theorem}
\newtheorem{prop}{Proposition}[section]
\newtheorem{lem}[prop]{Lemma}
\newtheorem{cor}[prop]{Corollary}

\newtheorem{rmk}[prop]{Remark}

\newcommand {\R} {\mathbb{R}} 
 \newcommand {\N} {\mathbb{N}}
\newcommand {\C} {\mathbb{C}} 
\newcommand {\p} {\partial}

\newcommand {\D} {\Delta}

\newcommand {\supp} {\text{supp}}

\newcommand{\tu}{\tilde{u}}

\title{On the Fractional Landis Conjecture}

\author[R\"uland]{Angkana R\"uland}
\address{Max-Planck-Institute for Mathematics in the Sciences, Inselstr. 22, 04103 Leipzig, Germany}
\email{rueland@mis.mpg.de}
\author[Wang]{Jenn-Nan Wang}
\address{Institute of Applied Mathematical Sciences, NCTS, National Taiwan University \\
Taipei 106, Taiwan}
\email{jnwang@math.ntu.edu.tw}
\thanks{Wang is supported in part by MOST 105-2115-M-002-014-MY3.}
\begin{document}

\begin{abstract}
In this paper we study a Landis-type conjecture for fractional Schr\"odinger equations of fractional power $s\in(0,1)$ with potentials. We discuss both the cases of differentiable and non-differentiable potentials. On the one hand, it turns out for \emph{differentiable} potentials with some a priori bounds, if a solution decays at a rate $e^{-|x|^{1+}}$, then this solution is trivial. On the other hand, for $s\in(1/4,1)$ and merely bounded \emph{non-differentiable} potentials, if a solution decays at a rate $e^{-|x|^\alpha}$  with $\alpha>4s/(4s-1)$, then this solution must again be trivial. Remark that when $s\to 1$, $4s/(4s-1)\to 4/3$ which is the optimal exponent for the standard Laplacian. For the case of non-differential potentials and $s\in(1/4,1)$, we also derive a quantitative estimate mimicking the classical result by Bourgain and Kenig.   
\end{abstract}

\maketitle

\section{Introduction}
\label{sec:intro}

In this work, we study a Landis-type conjecture for the fractional Schr\"odinger equation, 
\begin{equation}
\label{eq:eq_main}
((-\D)^s + q) u = 0\;\mbox{ in }\; \R^n
\end{equation}
with $s\in (0,1)$ and 
\begin{equation}\label{qbound}
|q(x)|\le 1.
\end{equation}
Roughly speaking, we are interested in the maximal vanishing rate of solutions to this equation at infinity. For $s=1$, in Section 3.5 in \cite{KL88a} V.A. Kondratev and E.M. Landis conjectured that if $|q(x)|\le 1$ and $|u(x)|\le C_0$ satisfies $|u(x)|\le \exp(-C|x|^{1+})$, then $u\equiv 0$. The Landis conjecture was disproved by Meshkov \cite{M91}, who constructed a potential $q$ and a nontrivial $u$ with $|u(x)|\le C\exp(-C|x|^{\frac 43})$. He also showed that if $|u(x)|\le C\exp(-C|x|^{\frac 43+})$, then $u\equiv 0$. In their seminal work, Bourgain and Kenig \cite{BK05} derived a quantitative form of Meshkov's result in their resolution of Anderson localization for the Bernoulli model in higher dimensions. It should be pointed out that in Meshkov's counterexample both $q$ and $u$ are \emph{complex}-valued functions. In other words, the exponent $4/3$ is optimal in the complex case (which corresponds to the situation of systems). The proof in \cite{BK05} is based on the Carleman method. In the spirit of the Carleman method, several extensions have been made in \cite{CS99,Davey14,DZ17,DZ18,LW14}, which also take singular drift coefficients and potentials into account.  

In view of Meshkov's counterexample, Kenig \cite{K05-06} refined the Landis conjecture and asked whether this conjecture is true for \emph{real}-valued potentials and solutions. In 2005, Kenig, Silvestre and the second author \cite{KSW15} confirmed the Landis conjecture (in a quantitative form) when $n=2$ and $q\ge 0$.  This result was later extended to the more general situation with $\D$ being replaced by any second order elliptic operator \cite{DKW17}. In the very recent preprint \cite{DKW18}, this is further improved by also allowing for (exponentially) small negative contributions in the potential. The Landis conjecture in the real case with $n=1$ was recently studied by Rossi \cite{Rossi18}. 

The main theme of this paper is to investigate a Landis-type conjecture for fractional Schr\"odinger equations. We will consider both qualitative and quantitative estimates when the potentials are either differentiable or simply bounded. Similar to the original Landis conjecture (for the Laplace operator), we are concerned with the maximal decay rate of non-trivial solutions. We are especially interested in understanding how the decay rate depends on the fractional power $s \in (0,1)$.  Detailed statements of our results are described below.

%In the sequel, we are interested in the maximal vanishing rate of solutions to this equation at infinity. For $s=1$ Landis conjectured that for the solution $u$ to be non-trivial it cannot vanish super-exponentially \cite{KL88}. While this is still a largely open question, we here address the situation in which $q$ satisfies additional regularity assumptions. Moreover, we also study both qualitatively and quantitatively the behaviour of solutions to the analogous fractional Schr\"odinger equation with complex potential. 

% ref Meshkov 

\subsection{Qualitative estimates}

We first discuss the qualitative behaviour of solutions to \eqref{eq:eq_main}.
Here we will show that if the potential $q$ satisfies an additional regularity estimate, then independently of the value of $s\in (0,1)$ there are no super-exponentially decaying solutions to \eqref{eq:eq_main}.

\begin{thm}
\label{thm:Landis_Diff}
Let $s\in (0,1)$ and assume that $u\in H^s(\R^n)$ is a solution to \eqref{eq:eq_main} such that $q$ satisfies $q\in C^1(\R^n)$, \eqref{qbound}, and in addition
\begin{equation}\label{grad}
|x\cdot \nabla q(x)|\leq 1
\end{equation}
holds. Suppose that $u$ further satisfies the following decay behaviour: there exists $\alpha >1 $ such that
\begin{align}
\label{eq:decay_1}
\int\limits_{\R^n} e^{|x|^{\alpha}} |u|^2 dx \leq C < \infty.
\end{align}
Then $u \equiv 0$.
\end{thm}

For $s=1$, a similar qualitative estimate as in Theorem~\ref{thm:Landis_Diff} with a differentiable potential satisfying \eqref{qbound} and \eqref{grad} was proved by Meshkov \cite{M92}. Without the additional regularity result on $q$, it is still possible to prove a qualitative decay result. However, as our argument for this does not distinguish between the real and complex situation, the obtained decay deteriorates.

\begin{thm}
\label{thm:Landis_No_Diff}
Let $s\in (1/4,1)$ and assume that $u\in H^{s}(\R^n,\C)$ is a  solution to \eqref{eq:eq_main}. Suppose that \eqref{qbound} holds and $u$ further satisfies the following decay behaviour: there exists $\alpha> \frac{4s}{4s-1}$ such that 
\begin{align}
\label{eq:decay_4_3}
\int\limits_{\R^n} e^{|x|^{\alpha}} |u|^2 dx \leq C < \infty.
\end{align}
Then $u \equiv 0$.
\end{thm}

We emphasize that as $s\rightarrow 1$ in both of our main results, the identified critical decay exponents correspond to the ones from the case $s=1$. Moreover, in the first result, Theorem \ref{thm:Landis_Diff}, the critical decay rate does \emph{not} depend on the value of $s\in(0,1)$. It is thus natural to ask whether the derived decay exponents are optimal or rather an artifact of our argument.
Let us comment on this. For \emph{real-valued problems} (i.e. scalar equations), we expect that the exponential decay (independent of the value of $s\in (0,1)$) as the critical decay behaviour is sharp. Indeed, as in \cite{KSW15} and \cite{BK05} it is possible to relate the decay behaviour at infinity to the local maximal vanishing rate at zero (if growth conditions are assumed, which are necessary due to the global character of the problem). Analogous arguments as in the classical case $s=1$ would lead to the conjecture that when considering 
\begin{itemize}
\item the equation \eqref{eq:eq_main} with potentials of the size $|q(x)|\leq M$ (instead of $|q(x)|\leq 1$), 
\item and solutions $u(x)$ which satisfy the growth bounds $\|u\|_{L^{\infty}(\R^n)}\leq C_0$ and $\|u\|_{L^{\infty}(B_1)}\geq 1$,
\end{itemize}
a lower bound of the form
\begin{equation}\label{bound1}
\|u\|_{L^{\infty}(B_r)} \geq r^{C M^{\frac{1}{2s}}},
\end{equation}
holds. Here $C=C(C_0)>0$ and $r\in (0,r_0)$ for some sufficiently 
small constant $r_0>0$. Results of this flavour have been 
proved for eigenfunctions or equations with differentiable potentials (with dependences on the $C^1$ norm of the potentials, that is, $M$ in \eqref{bound1} is the size of $\|q\|_{C^1(\R^n)}$) in \cite{R17, Z15} (on compact manifolds or bounded domains, respectively). 
For the spectral fractional Laplacian and its eigenfunctions on compact manifolds these dependences are indeed immediate consequences from the corresponding ones of the Laplacian.

We recall that $\nabla q$ satisfies \eqref{grad}. Taking this expected quantitative maximal rate of vanishing \eqref{bound1} for granted and relying on a scaling argument as in \cite{BK05, KSW15}, i.e. considering solutions $u_{R,x_0}(x)$ for $x_0 \in B_{R}\setminus B_{R/2}$ of the correspondingly rescaled version of \eqref{eq:eq_main}, then 
suggest global lower bounds of the type
\begin{align*}
\inf\limits_{|z_0|=R} \sup\limits_{|z-z_0|<1} |u(z)| \geq \exp(-C R \log(R)) \mbox{ for } R \gg 1,
\end{align*}
where $C=C(C_0)>0$. Hence, this strongly suggests that for the class of potentials under consideration, Theorem \ref{thm:Landis_Diff} is essentially sharp (possibly up to logarithmic contributions). However, we further believe that as in \cite{KSW15}, at least under sign conditions on the potential and in one dimension (which on the level of the Cafferelli-Silvestre extension corresponds to the two-dimensional setting from \cite{KSW15} in which complex analysis tools are available), it might be possible to reduce the necessary \emph{regularity} for $q$ to $L^{\infty}$ regularity.

In view of Meshkov's example \cite{M91}, in the case of Theorem \ref{thm:Landis_No_Diff}, at least the growth behaviour for $s\rightarrow 1$ is expected to be optimal. As in the case $s=1$, the Carleman estimate of Theorem \ref{thm:Carl_No_Diff} which lies at the core of the argument for Theorem \ref{thm:Landis_No_Diff} is of perturbative character. It hence does \emph{not} distinguish between the real vs the complex, i.e. the scalar vs the systems cases. It would be interesting to investigate whether also for $s\in (1/4,1)$ there are Meshkov-type examples saturating the proposed exponents from Theorem \ref{thm:Landis_No_Diff}. We remark that the restriction to $s\in (1/4,1)$ seems necessary as long as we only consider radial weight functions in our Carleman estimates (due to the subelliptic nature of these estimates). We seek to prove the results of Theorems \ref{thm:Landis_Diff} and \ref{thm:Landis_No_Diff} by combining elliptic estimates with Carleman estimates.

\subsection{Quantitative estimates}
In this subsection, we present a quantitative version of Theorem \ref{thm:Landis_No_Diff}. Note that here $q\in L^\infty(\R^n)$ and that \eqref{qbound} is satisfied. 

\begin{thm}
\label{thm:quant}
Let $s\in (1/4,1)$ and assume that $u\in H^{s}(\R^n,\C)$ is a solution to \eqref{eq:eq_main}. Suppose that $u$ further satisfies $\|u\|_{L^2(\{|x|<1\})}=1$
and that there exists a constant $C_0>0$ such that
\[
\|u\|_{L^{\infty}(\R^n)}\leq C_0.
\]
Then there exists a constant $C=C(n,s,C_0)>0$ such that for $R>0$ large
\begin{align*}
\inf\limits_{|x_0|=R} \|u\|_{L^{{\infty}}(\{|x-x_0|<1\})}\geq Ce^{-C R^{\frac{4s}{4s-1}}\log R} .
\end{align*}
\end{thm}

We remark that by virtue of the ellipticity of the problem, it does not really matter in which topology one works for the lower bounds. For instance, it would have equally been possible to derive similar results under the assumption that $|u(0)|=1$.

This result is similar in flavour to the decay estimates in \cite{BK05}. However both the arguments in the qualitative and the quantitative settings involve new intricacies and technical challenges due to the nonlocal character of the equations at hand. Since lower bound estimates for nonlocal equations pose serious difficulties, as in various other works on (quantitative) unique continuation \cite{FF14, FF15, BG17,G17,Rue15,R17,Rue17,RS17}, we opt for working with the Caffarelli-Silvestre extension \cite{CS07}, c.f. \eqref{eq:CS}, \eqref{eq:CS_eq_main} in Section \ref{sec:aux}, instead of dealing with the nonlocal equation \eqref{eq:eq_main} directly. This allows us to investigate a local (degenerate) elliptic equation. It however comes at the expense of having to study this equation in $n+1$ instead of $n$ dimensions. In the additional dimension, the control on the solution can hence only be derived through the equation. This implies that we always have to transfer information from the boundary to the bulk and vice versa. In the qualitative estimates, we for instance have to show that (exponential) decay on the boundary implies (exponential) decay in the bulk. Similarly, we have to transfer upper and lower bounds in the quantitative results from the boundary into the bulk and vice versa. This poses non-trivial challenges, which however are overcome by an ingredient which was used in \cite{RS17}:
As one of our key tools which allows us to switch between the bulk and the boundary we rely on a boundary-bulk interpolation inequality, c.f. Proposition \ref{prop:small_prop_bound}.

\subsection{Organization of the article}
This paper is organized as follows. In Section~\ref{sec:aux}, we will collect several preliminary results that are needed in the proofs of the main theorems. We then derive Carleman estimates for the fractional Laplacian $(-\D)^s$ with differentiable or non-differentiable potentials in Section~\ref{eq:Carl}. The proofs of the qualitative estimates, Theorem~\ref{thm:Landis_Diff} and \ref{thm:Landis_No_Diff}, are given in Section~\ref{sec4}.  Finally, in Section~\ref{sec5}, we will prove Theorem \ref{thm:quant}.

\section{Auxiliary Results}
\label{sec:aux}

\subsection{The Cafferelli-Silvestre extension}
\label{sec:Caff_Silv}

In the sequel, it will be convenient to localize the problem at hand. This will be achieved by means of the Caffarelli-Silvestre extension \cite{CS07} which allows us to address the problem at hand by relying on tools for \emph{local} equations. To this end, for $s\in (0,1)$ and $u\in H^{s}(\R^n)$ we consider a solution $\tilde{u}\in \dot{H}^{1}(\R^{n+1}_+, x_{n+1}^{1-2s}):=\{v: \R^{n+1}_+ \rightarrow \R: \ \int\limits_{\R^{n+1}_+} x_{n+1}^{1-2s} |\nabla v|^2 dx \leq C < \infty \}$ of the degenerate elliptic equation (whose weight however still is in the Muckenhoupt class),
\begin{align}\label{eq:CS}
\begin{split}
\nabla \cdot x_{n+1}^{1-2s} \nabla \tilde{u} & = 0 \mbox{ in } \R^{n+1}_+,\\
\tilde{u} & = u \mbox{ on } \R^{n} \times \{0\}.
\end{split}
\end{align}
We recall that by the observations in \cite{CS07},
\begin{align*}
(-\D)^s u (x) = c_{n,s}\lim\limits_{x_{n+1} \rightarrow 0} x_{n+1}^{1-2s} \p_{n+1} \tilde{u}(x)
\end{align*}
for some constant $c_{n,s}\neq 0$.
In particular, the equation \eqref{eq:eq_main} can be reformulated as the local, degenerate elliptic equation
\begin{align}
\label{eq:CS_eq_main}
\begin{split}
\nabla \cdot x_{n+1}^{1-2s} \nabla \tilde{u} & = 0 \mbox{ in } \R^{n+1}_+,\\
\tilde{u} & = u \mbox{ on } \R^n \times \{0\},\\
c_{n,s}\lim\limits_{x_{n+1} \rightarrow 0} x_{n+1}^{1-2s}\p_{n+1}\tilde{u} & = q u \mbox{ on } \R^{n} \times \{0\}.
\end{split}
\end{align}
This however comes at the expense of adding a new variable, in which we have to infer control by exploiting the equation. 

When dealing with this equation, it will be convenient to also introduce the following notations for the underlying domains and the related weighted function spaces. For $\Omega \subset \R^{n+1}_+$, $x_0 \in \R^{n+1}_+$, $r,R>0$, we denote
\begin{align*}
B_{r}^+(x_0)&:= \{x\in \R^{n+1}_+: \ |x-x_0|\leq r\}, \
B_{r}'(x_0):=\{x\in \R^{n}\times \{0\}: \ |x-x_0|\leq r\},\\ 
B_r^+&:=B_r^+(0), \ B_r':= B_r'(0),\\
A_{r,R}^+&:=\{x\in \R^{n+1}_+: \  r\leq |x|\leq R\}, \ 
A_{r,R}':= \{x\in \R^n \times \{0\}: \ r\leq |x|\leq R \},\\
H^{1}(\Omega,x_{n+1}^{1-2s})
&:= \{v: \Omega \cap \R^{n+1}_+ \rightarrow \R: \ \int\limits_{\Omega \cap \R^{n+1}_+} x_{n+1}^{1-2s} (|v|^2 + |\nabla v|^2) dx \leq C < \infty \}, \\ 
H^{1}(S^{n}_+, \theta_n^{1-2s})&:= \{v: S^{n}_+ \rightarrow \R: \ \int\limits_{S^{n}_+} \theta_{n}^{1-2s} (|v|^2 + |\nabla_{S^n_+} v|^2) d\theta \leq C < \infty \}, \ \theta_n=\frac{x_{n+1}}{|x|}.
\end{align*}

As an important elliptic estimate which we will be using frequently we recall Caccioppoli's inequality.

\begin{lem}
\label{lem:Cacc}
Let $s\in (0,1)$ and $\tilde{u}\in H^{1}(B_{4r}^+,x_{n+1}^{1-2s})$ be a solution to \eqref{eq:CS}. Then, there exists $C=C(n,s)>0$ such that
\begin{align*}
\|x_{n+1}^{\frac{1-2s}{2}} \nabla \tilde{u}\|_{L^2(B_{r}^+)}
\leq C\left( r^{-1}\|x_{n+1}^{\frac{1-2s}{2}} \tilde{u}\|_{L^2(B_{2r}^+)} + \|u\|_{L^2(B_{2r}')}^{\frac{1}{2}} \|\lim\limits_{x_{n+1}\rightarrow 0} x_{n+1}^{1-2s} \p_{n+1}\tilde u\|_{L^2(B_{2r}')}^{\frac{1}{2}} \right).
\end{align*}
\end{lem}

\begin{proof}
The proof follows as for instance in Lemma 4.5 in \cite{RS17}, where however the boundary terms are estimated by an $L^2-L^2$ estimate instead of an $H^{s}-H^{-s}$ estimate.
\end{proof}

\subsection{Boundary decay implies bulk decay}
\label{sec:boundary_bulk}

In order to deal with the original nonlocal problem \eqref{eq:eq_main}, we seek to apply methods which were developed for (quantitative) unique continuation results for the local equation \eqref{eq:CS_eq_main}. Hence, we first translate the decay behaviour that is valid on $\R^n$ to decay behaviour which also holds on $\R^{n+1}_+$. To this end, we heavily rely on interior and boundary three balls estimates for the degenerate elliptic equation \eqref{eq:CS_eq_main}.

\begin{prop}
\label{prop:boundary_bulk}
Let $s\in(0,1)$ and $u\in H^{s}(\R^n)$ be a solution to \eqref{eq:eq_main}.
Assume that \eqref{qbound} holds and there exist constants $C,\beta \geq 1$ such that
\begin{equation}
\label{ubound}
\|e^{|x|^{\beta}/2}u\|_{L^2(\R^n)}\leq C .
\end{equation}   
Then, there exist constants $C_1,c_1>0$ such that for all $x=(x',x_{n+1})\in \R^{n+1}_+$ the Caffarelli-Silvestre extension $\tilde{u}(x)$ satisfies 
\begin{align*}
|\tilde{u}(x',x_{n+1})| \leq C_1 e^{-c_1 |(x',x_{n+1})|^{\beta}}. 
\end{align*}
\end{prop}

In order to infer the claimed interior decay, we rely on propagation of smallness estimates. Here we make use of two types of propagation of smallness results: The first being an interior propagation of smallness while the second one is a boundary-bulk propagation of smallness estimate. In order to use tools from the quantitative analysis of elliptic equations, in the sequel, we view \eqref{eq:eq_main} in terms of its Caffarelli-Silvestre extension \eqref{eq:CS_eq_main}. 

\begin{prop}
\label{prop:small_prop_int}
Let $s\in (0,1)$ and $\tilde u\in H^{1}(B_{4}^+, x_{n+1}^{1-2s})$ be a solution to \eqref{eq:CS}. Assume that $r\in (0,1)$ and $\overline{x}_0 = (\overline{x}_0', 5 r) \in B_{2}^+$. Then, there exists $\alpha = \alpha(n,s) \in (0,1)$ such that
\begin{align*}
\|\tilde{u}\|_{L^{\infty}(B_{2r}^+(\overline{x}_0))}
\leq C\|\tilde{u}\|_{L^{\infty}(B_{r}^+(\overline{x}_0))}^{\alpha} \|\tilde{u}\|_{L^{\infty}(B_{4r}^+(\overline{x}_0))}^{1-\alpha}.
\end{align*}
\end{prop}

\begin{proof}
As $(x_0)_{n+1} =5r $, this follows from a standard interior $L^2$ three balls estimate (c.f. Proposition 5.4 in \cite{RS17}) together with $L^2-L^{\infty}$ estimates for uniformly elliptic equations.
\end{proof}

\begin{prop}
\label{prop:small_prop_bound}
Let $s\in (0,1)$ and let $\tilde{u}\in H^{1}(\R^{n+1}_+, x_{n+1}^{1-2s})$ be a solution to \eqref{eq:CS_eq_main} with $q\in L^{\infty}(\R^n)$. Assume that $x_0\in\R^n\times\{0\}$. Then,
\begin{itemize}
\item[(a)] there exist $\alpha = \alpha(n,s) \in (0,1)$ and $c=c(n,s)\in (0,1)$ such that
\begin{align*}
%\label{eq:combi}
\begin{split}
&\|x_{n+1}^{\frac{1-2s}{2}} \tu\|_{L^2(B_{cr}^+(x_0))}\\
\leq &C
\left(\|x_{n+1}^{\frac{1-2s}{2}} \tu\|_{L^2(B_{16r}^+(x_0))}+r^{1-s} \|u\|_{L^2(B_{16r}'(x_0))}\right)^{\alpha} \times\\
& \quad \quad \times
\left(r^{{s+1}}\|\lim\limits_{x_{n+1}\rightarrow 0} x_{n+1}^{1-2s} \p_{n+1} \tu\|_{L^{2}(B_{16r}'(x_0))}+r^{{1-s}} \|u\|_{L^2(B_{16r}'(x_0))}\right)^{1-\alpha}\\
&+C
\left(\|x_{n+1}^{\frac{1-2s}{2}} \tu\|_{L^2(B_{16r}^+(x_0))}+r^{{1-s}} \|u\|_{L^2(B_{16r}'(x_0))}\right)^{\frac{2s}{1+s}}\times \\
& \quad \quad \times \left(r^{{s+1}}\|\lim\limits_{x_{n+1}\rightarrow 0} x_{n+1}^{1-2s} \p_{n+1} \tu\|_{L^{2}(B_{16r}'(x_0))}+r^{{1-s}} \|u\|_{L^2(B_{16r}'(x_0))}\right)^{\frac{1-s}{1+s}}.
\end{split}
\end{align*}

\item[(b)] there exist $\alpha = \alpha(n,s) \in (0,1)$ and $c=c(n,s)\in (0,1)$ such that
\begin{align}
\label{eq:combi_up_Linf}
\begin{split}
\| \tu\|_{L^{\infty}(B_{\frac{cr}{2}}^+)}
\leq&C
r^{-\frac{n}{2}}\left(r^{s-1}\|x_{n+1}^{\frac{1-2s}{2}} \tu\|_{L^2(B_{16r}^+)}+ \|u\|_{L^2(B_{16r}')}\right)^{\alpha}
\left(r^{2s}\|qu\|_{L^{2}(B_{16r}')}+ \|u\|_{L^2(B_{16r}')}\right)^{1-\alpha}\\
&+C
r^{-\frac{n}{2}}\left(r^{s-1}\|x_{n+1}^{\frac{1-2s}{2}} \tu\|_{L^2(B_{16r}^+)}+ \|u\|_{L^2(B_{16r}')}\right)^{\frac{2s}{1+s}}
\left(r^{2s}\|qu\|_{L^{2}(B_{16r}')}+ \|u\|_{L^2(B_{16r}')}\right)^{\frac{1-s}{1+s}}\\
&+ Cr^{-\frac n2}r^s\|qu\|_{L^2(B_{16r}')}^{\frac 12}\|u\|_{L^2(B_{16r}')}^{\frac{1}{2}}.
\end{split}
\end{align}
\end{itemize}
\end{prop}

\begin{proof}
The proof relies on a splitting argument and the boundary-bulk interpolation estimates from Propositions 5.10-5.12 (also Proposition 5.6) in \cite{RS17}. %(treat Dirichlet and Neumann data separately).
In order to infer the claim, we argue in two steps, first deriving a suitable $L^2$ estimate and then upgrading this to an $L^{\infty}$ estimate. By scaling, it suffices to prove the estimate for $r=1$. Without loss of generality, we can take $x_0=0$.\\

\emph{Step 1: The $L^2$ estimate.}
For the $L^2$ estimate we rely on Propositions 5.10-5.12 in \cite{RS17}. Here we distinguish between the cases $s\in [\frac{1}{2},1)$ and $s\in(0,\frac{1}{2})$.\\

\emph{Step 1a: The case $s \in [\frac{1}{2},1)$.}
In order to invoke the estimate from \cite{RS17}, we split our solution $u$ into two parts $\tu=u_1+u_2$. The function $u_1$ deals with the Dirichlet data
\begin{align*}
\nabla \cdot x_{n+1}^{1-2s} \nabla u_1 & = 0 \mbox{ in } \R^{n+1}_+,\\
u_1 & = \zeta u \mbox{ in } \R^{n} \times \{0\}.
\end{align*}
Here $\zeta \in C^{\infty}_0(B_{16}')$ is a smooth cut-off function, which is equal to one in $B_{8}'$. We will estimate $u_1$ by bounds on the Caffarelli-Silvestre extension. The function $u_2=\tu-u_1$ in turn is admissible in  Propositions 5.10-5.12 in \cite{RS17}, i.e., $u_2|_{B_8'}=0$.

We begin with the estimate for $u_1$:
Invoking Lemma 4.2 in \cite{RS17}, we obtain the bound
\begin{align}
\label{eq:L2u1}
\|x_{n+1}^{\frac{1-2s}{2}} u_1\|_{L^2(\R^{n+1}_+)}
\leq C \|\zeta u \|_{H^{s-1}(\R^n)} \leq C \|\zeta u\|_{L^2(\R^n)} \leq C \|u\|_{L^2(B_{16}')}.
\end{align}
The estimate for $u_2$ follows from Proposition 5.10 in \cite{RS17}. The result assert that for each $s\in (\frac{1}{2},1)$ there exists a constant $c=c(s,n)\in (0,1)$ and $\alpha = \alpha(s,n)\in (0,1)$ such that
\begin{equation}\label{eq:L2u2}
%\begin{aligned}
\|x_{n+1}^{\frac{1-2s}{2}} u_2\|_{L^2(B_{c}^+)}
\leq C \|x_{n+1}^{\frac{1-2s}{2}} u_2\|_{L^2(B_{2}^+)}^{\alpha}
\|\lim\limits_{x_{n+1}\rightarrow 0} x_{n+1}^{1-2s} \p_{n+1} u_2\|_{L^2(B_{2}')}^{1-\alpha}.
%&{\color{red}+C\|\lim\limits_{x_{n+1}\rightarrow 0} x_{n+1}^{1-2s} \p_{n+1} u_2\|_{L^2(B_{2}')}}.
%\end{aligned}
\end{equation}
We modify this by interpolation in order to obtain an estimate where the normal derivative of $u_2$ is measured in the $H^{-2s}$ norm. To this end, we note that for any $w$ by interpolation and the characterization of the trace map (c.f. Step 1 in the proof of Proposition 5.11 in \cite{RS17}):
\begin{align}
\label{eq:bulk_boundary_interpol}
\begin{split}
\|w\|_{L^2(\R^n)}
&\leq C \| w\|_{H^{1-s}(\R^n)}^{\frac{2s}{1+s}}\|w\|_{H^{-2s}(\R^n)}^{\frac{1-s}{1+s}}\\
&\leq C\left(\|x_{n+1}^{\frac{2s-1}{2}}w \|_{L^2(\R^{n+1}_+)}+\|x_{n+1}^{\frac{2s-1}{2}} \nabla w \|_{L^2(\R^{n+1}_+)}\right)^{\frac{2s}{1+s}} \|w\|_{H^{-2s}(\R^n)}^{\frac{1-s}{1+s}} \\
&\leq C\left(\mu^{1-s}\left(\|x_{n+1}^{\frac{2s-1}{2}}w \|_{L^2(\R^{n+1}_+)}+\|x_{n+1}^{\frac{2s-1}{2}} \nabla w \|_{L^2(\R^{n+1}_+)}\right) + \mu^{-2s}\| w\|_{H^{-2s}(\R^n)} \right). 
\end{split}
\end{align}
Applying this to $w= \eta x_{n+1}^{1-2s} \p_{n+1} u_2$, where $\eta$ is a smooth, radial cut-off function which is equal to one on $B_{2}^+$ and vanishes outside of $B_{4}^+$ gives us
\begin{equation}\label{eq:interpol_2s}
\begin{aligned}
&\|\lim\limits_{x_{n+1}\rightarrow 0} x_{n+1}^{1-2s} \p_{n+1} u_2\|_{L^2(B_{2}')}\\
\leq &C\left(\mu^{1-s}\left(\|x_{n+1}^{\frac{1-2s}{2}}\partial_{n+1} u_2\|_{L^2(B_{4}^+)}+\|x_{n+1}^{\frac{2s-1}{2}} \nabla (\eta x_{n+1}^{1-2s} \p_{n+1} u_2)\|_{L^2(\R^{n+1}_+)}\right)\right.\\
&\left. + \mu^{-2s} \|\lim_{x_{n+1}\rightarrow 0}\eta x_{n+1}^{1-2s} \p_{n+1} u_2\|_{H^{-2s}(\R^n)}\right).
\end{aligned}
\end{equation}

Similar to the proof of Proposition 5.11 in \cite{RS17}, we now estimate each term on the right hand side of \eqref{eq:interpol_2s}.
The last term gives us
\begin{equation}\label{0807-1}
 \|\lim_{x_{n+1}\rightarrow 0}\eta x_{n+1}^{1-2s} \p_{n+1} u_2\|_{H^{-2s}(\R^n)}\le C\|\lim_{x_{n+1}\rightarrow 0}x_{n+1}^{1-2s} \p_{n+1} u_2\|_{H^{-2s}(B_{8}')}.
 \end{equation}
Applying Caccioppoli's inequality in Lemma~\ref{lem:Cacc} (with zero Dirichlet condition) implies
\begin{equation}\label{0808-1}
\|x_{n+1}^{\frac{1-2s}{2}}\partial_{n+1} u_2\|_{L^2(B_{4}^+)}\le C\|x_{n+1}^{\frac{1-2s}{2}}u_2\|_{L^2(B_{8}^+)}.
\end{equation}
It remains to estimate the second term on the right hand side of \eqref{eq:interpol_2s}. Indeed, for the resulting bulk term we have
\begin{equation}\label{0808-2}
\begin{aligned}
\|x_{n+1}^{\frac{2s-1}{2}} \nabla (\eta x_{n+1}^{1-2s} \p_{n+1} u_2)\|_{L^2(\R^{n+1}_+)}
&\leq \|x_{n+1}^{\frac{1-2s}{2}} (\nabla \eta) (\p_{n+1} u_2)\|_{L^2(\R^{n+1}_+)}+ \|x_{n+1}^{\frac{1-2s}{2}} \eta \nabla' \p_{n+1} u_2\|_{L^2(\R^{n+1}_+)} \\
& \quad + \|x_{n+1}^{\frac{2s-1}{2}} \eta \p_{n+1} x_{n+1}^{1-2s} \p_{n+1} u_2\|_{L^2(\R^{n+1}_+)}\\
&\leq \|x_{n+1}^{\frac{1-2s}{2}} (\p_{n+1} u_2)\|_{L^2(B_{4}^+)}  + \|x_{n+1}^{\frac{1-2s}{2}} \nabla' \p_{n+1} u_2\|_{L^2(B_{4}^+)} \\
& \quad + \|x_{n+1}^{\frac{1-2s}{2}} \eta \Delta' u_2\|_{L^2(\R^{n+1}_+)}\\
& \leq C \|x_{n+1}^{\frac{1-2s}{2}} u_2\|_{L^2(B_{8}^+)}.
\end{aligned}
\end{equation}
Here we first used the triangle inequality, then the support condition for $\eta$ and the equation for $u_2$ and finally applied Caccioppoli's inequality (twice for the last two terms, noting that $\nabla' u_2$ solves a similar problem).

%For the tangential derivate we observe that
%\begin{align*}
%\|\lim\limits_{x_{n+1}\rightarrow 0} x_{n+1}^{1-2s} \p_{n+1} (\eta u_2)\|_{H^{-2s}(\R^n)}
%&= \|\eta \lim\limits_{x_{n+1}\rightarrow 0} x_{n+1}^{1-2s} \p_{n+1}  u_2\|_{H^{-2s}(\R^n)}\\
%& \leq \|\lim\limits_{x_{n+1}\rightarrow 0} x_{n+1}^{1-2s} \p_{n+1} u_2 \|_{H^{-2s}(B_{32}')}. 
%\end{align*}
%Here we first used the radial dependence of $\eta$ and then a duality argument and the fact that $\eta$ is a bounded multiplier in $H^{2s}$.
%Combining the bulk and tangential estimate hence results in \eqref{eq:interpol_2s}.

Substituting \eqref{0807-1}-\eqref{0808-2} into \eqref{eq:interpol_2s} and optimizing the resulting estimate in $\mu>0$ gives
\begin{align*}
\|\lim\limits_{x_{n+1}\rightarrow 0} x_{n+1}^{1-2s} \p_{n+1} u_2\|_{L^2(B_{2}')}
\leq C \|x_{n+1}^{\frac{1-2s}{2}} u_2\|_{L^2(B_{8}^+)}^{\frac{2s}{1+s}}\|\lim\limits_{x_{n+1}\rightarrow 0} x_{n+1}^{1-2s} \p_{n+1}u_2\|_{H^{-2s}(B_{8}')}^{\frac{1-s}{1+s}}.
\end{align*}
Inserting this into \eqref{eq:L2u2} leads to
\begin{equation}
\label{eq:L2u2_a}
\|x_{n+1}^{\frac{1-2s}{2}} u_2\|_{L^2(B_{c}^+)}
\leq C \|x_{n+1}^{\frac{1-2s}{2}} u_2\|_{L^2(B_{8}^+)}^{\tilde{\alpha}}
\|\lim\limits_{x_{n+1}\rightarrow 0} x_{n+1}^{1-2s} \p_{n+1} u_2\|_{H^{-2s}(B_{8}')}^{1-\tilde{\alpha}}
\end{equation}
where $\tilde{\alpha}= \frac{1-s}{1+s}\alpha + \frac{2s}{1+s}$. By slight abuse of notation, in the sequel, we simply drop the tilde. Combining the two bounds \eqref{eq:L2u1}, \eqref{eq:L2u2_a} and
\[
\|\lim\limits_{x_{n+1}\rightarrow 0} x_{n+1}^{1-2s} \p_{n+1} u_1\|_{H^{-2s}(B_{16}')}
\leq \|(-\D)^s u_1\|_{L^2(\R^n)}
\leq \|u_1\|_{L^2(\R^n)}
\leq C\|u\|_{L^2(B_{16}')},
\]
and applying the triangle inequality leads to 
\begin{align}
\label{eq:combi}
\begin{split}
&\|x_{n+1}^{\frac{1-2s}{2}} \tu\|_{L^2(B_{c}^+)}\\
\leq &C 
\left(\|x_{n+1}^{\frac{1-2s}{2}} \tu\|_{L^2(B_{16}^+)}+ \|u\|_{L^2(B_{16}')}\right)^{\alpha}
\left(\|\lim\limits_{x_{n+1}\rightarrow 0} x_{n+1}^{1-2s} \p_{n+1} \tu\|_{H^{-2s}(B_{16}')}+ \|u\|_{L^2(B_{16}')}\right)^{1-\alpha}\\
%&+C
%\left(\|x_{n+1}^{\frac{1-2s}{2}} \tu\|_{L^2(B_{16}^+)}+ \|u\|_{L^2(B_{16}')}\right)^{\frac{2s}{1+s}}
%\left(\|\lim\limits_{x_{n+1}\rightarrow 0} x_{n+1}^{1-2s} \p_{n+1} \tu\|_{H^{-2s}(B_{16}')}+ \|u\|_{L^2(B_{16}')}\right)^{\frac{1-s}{1+s}}\\
\leq &C
\left(\|x_{n+1}^{\frac{1-2s}{2}} \tu\|_{L^2(B_{16}^+)}+ \|u\|_{L^2(B_{16}')}\right)^{\alpha}
\left(\|\lim\limits_{x_{n+1}\rightarrow 0} x_{n+1}^{1-2s} \p_{n+1} \tu\|_{L^{2}(B_{16}')}+ \|u\|_{L^2(B_{16}')}\right)^{1-\alpha}.
\end{split}
\end{align}
This already implies the claim of (a). In order to exploit it for the proof of (b), we strengthen the estimate slightly.
By Caccioppoli's inequality of Lemma~\ref{lem:Cacc} (now with an $L^2$ estimate for the boundary contributions), we can further upgrade \eqref{eq:combi} to 
\begin{align}
\label{eq:combi_up0}
\begin{split}
&\|x_{n+1}^{\frac{1-2s}{2}} \tu\|_{L^2(B_{\tilde{c}}^+)}
+ \|x_{n+1}^{\frac{1-2s}{2}} \nabla \tu\|_{L^2(B_{\tilde{c}}^+)}\\
\leq &C
\left(\|x_{n+1}^{\frac{1-2s}{2}} \tu\|_{L^2(B_{16}^+)}+ \|u\|_{L^2(B_{16}')}\right)^{\alpha}
\left(\|\lim\limits_{x_{n+1}\rightarrow 0} x_{n+1}^{1-2s} \p_{n+1} \tu\|_{L^{2}(B_{16}')}+ \|u\|_{L^2(B_{16}')}\right)^{1-\alpha}\\
%&+C
%\left(\|x_{n+1}^{\frac{1-2s}{2}} \tu\|_{L^2(B_{16}^+)}+ \|u\|_{L^2(B_{16}')}\right)^{\frac{2s}{1+s}}
%\left(\|\lim\limits_{x_{n+1}\rightarrow 0} x_{n+1}^{1-2s} \p_{n+1} \tu\|_{L^{2}(B_{16}')}+ \|u\|_{L^2(B_{16}')}\right)^{\frac{1-s}{1+s}}.\\
&+ C\|\lim\limits_{x_{n+1}\rightarrow 0} x_{n+1}^{1-2s} \p_{n+1} \tu\|_{L^2(B_{16}')}^{\frac 12}\|u\|_{L^2(B_{16}')}^{\frac{1}{2}}
\end{split}
\end{align}
with $\tilde{c}=c/2$.\\

\emph{Step 1b: The case $s\in (0,\frac{1}{2})$.} The case $s\in (0,1/2)$ is similar as the case discussed above and relies on a splitting strategy. As above, the estimate for $u_1$ is a direct consequence of the boundary bulk estimates for the Caffarelli-Silvestre extension. Thus, the main remaining estimate is the derivation of the corresponding analogue of \eqref{eq:L2u2_a}. As in the proof of \eqref{eq:L2u2_a} this follows an application of the corresponding $L^2$ result from \cite{RS17} (Proposition 5.12) and interpolation. More precisely, Proposition 5.12 in \cite{RS17} implies that for some $c=c(n,s)>0$ and $\alpha = \alpha(n,s)\in (0,1)$ we have
\begin{align*}%\label{0808-5}
\begin{split}
&\|x_{n+1}^{\frac{1-2s}{2}} u_2 \|_{L^2(B_c^+)}\\
&\leq C (\|x_{n+1}^{\frac{1-2s}{2}} u_2\|_{L^2(B_{2}^+)}^{\alpha} \|\lim\limits_{x_{n+1}\rightarrow 0} x_{n+1}^{1-2s} \p_{n+1} u_2\|_{L^{2}(B_{2}')}^{1-\alpha} + \| \lim\limits_{x_{n+1}\rightarrow 0}x_{n+1}^{1-2s}\p_{n+1} u_2\|_{L^{2}(B_{2}')}).
\end{split}
\end{align*}
In order to pass from this estimate which involves an $L^2$ norm of the weighted Neumann data to an estimate which involves its $H^{-2s}$ norm, we apply the interpolation estimate \eqref{eq:bulk_boundary_interpol} as in the case $s\in (\frac{1}{2},1)$. With this estimate at hand, the analogues of \eqref{eq:combi} and \eqref{eq:combi_up0} then follow by combining the estimates of the splitting argument as above. Note that \eqref{eq:combi_up0} now becomes
\begin{align}
\label{eq:combi_up}
\begin{split}
&\|x_{n+1}^{\frac{1-2s}{2}} \tu\|_{L^2(B_{\tilde{c}}^+)}
+ \|x_{n+1}^{\frac{1-2s}{2}} \nabla \tu\|_{L^2(B_{\tilde{c}}^+)}\\
\leq &C
\left(\|x_{n+1}^{\frac{1-2s}{2}} \tu\|_{L^2(B_{16}^+)}+ \|u\|_{L^2(B_{16}')}\right)^{\alpha}
\left(\|\lim\limits_{x_{n+1}\rightarrow 0} x_{n+1}^{1-2s} \p_{n+1} \tu\|_{L^{2}(B_{16}')}+ \|u\|_{L^2(B_{16}')}\right)^{1-\alpha}\\
&+C
\left(\|x_{n+1}^{\frac{1-2s}{2}} \tu\|_{L^2(B_{16}^+)}+ \|u\|_{L^2(B_{16}')}\right)^{\frac{2s}{1+s}}
\left(\|\lim\limits_{x_{n+1}\rightarrow 0} x_{n+1}^{1-2s} \p_{n+1} \tu\|_{L^{2}(B_{16}')}+ \|u\|_{L^2(B_{16}')}\right)^{\frac{1-s}{1+s}}\\
&+ C\|\lim\limits_{x_{n+1}\rightarrow 0} x_{n+1}^{1-2s} \p_{n+1} \tu\|_{L^2(B_{16}')}^{\frac 12}\|u\|_{L^2(B_{16}')}^{\frac{1}{2}}.
\end{split}
\end{align}

\emph{Step 2: The $L^{\infty}$ estimate.}
In order to pass from the $L^2$-based bounds from step 1 to $L^{\infty}$ based estimates, we rely on an estimate due to Jin, Li, Xiong \cite{JLX11} (Proposition 2.4 (i), c.f. also Proposition 3.2 in \cite{FF14}), which states that under our conditions on $B_{1/2}^+$ it holds
\begin{align}
\label{eq:LinfL2}
\|\tilde{u}\|_{L^{\infty}(B_{1/2}^+)} \leq C (\|x_{n+1}^{\frac{1-2s}{2}} \tilde{u}\|_{L^{2}(B_1^+)} + \|x_{n+1}^{\frac{1-2s}{2}} \nabla \tilde{u}\|_{L^{2}(B_1^+)}).
\end{align}
Combining this with the estimate \eqref{eq:combi_up} and inserting the identity $c_{n,s}\lim\limits_{x_{n+1}\rightarrow 0} x_{n+1}^{1-2s} \p_{n+1} \tu = q u $ entails
\begin{align}
\label{eq:combi_up_Linf1}
\begin{split}
\| \tu\|_{L^{\infty}(B_{\frac{\tilde c}{2}}^+)}
\leq&C
\left(\|x_{n+1}^{\frac{1-2s}{2}} \tu\|_{L^2(B_{16}^+)}+ \|u\|_{L^2(B_{16}')}\right)^{\alpha}
\left(\|qu\|_{L^{2}(B_{16}')}+ \|u\|_{L^2(B_{16}')}\right)^{1-\alpha}\\
&+C
\left(\|x_{n+1}^{\frac{1-2s}{2}} \tu\|_{L^2(B_{16}^+)}+ \|u\|_{L^2(B_{16}')}\right)^{\frac{2s}{1+s}}
\left(\|qu\|_{L^{2}(B_{16}')}+ \|u\|_{L^2(B_{16}')}\right)^{\frac{1-s}{1+s}}\\
&+ C\|qu\|_{L^2(B_{16}')}^{\frac 12}\|u\|_{L^2(B_{16}')}^{\frac{1}{2}}.
\end{split}
\end{align}
Based on this, we can also derive a pure $L^\infty$ estimate.  Indeed, using H\"{o}lder's inequality together with the $L^2$ integrability of the function $(0,1) \ni t \mapsto t^{\frac{1-2s}{2}} \in \R$ for $s\in (0,1)$ on bounded domains results in
\begin{align*}
%\label{eq:combi_up_Linfa}
\begin{split}
\| \tu\|_{L^{\infty}(B_{\frac{\tilde c}{2}}^+)}
\leq&C (1+\|q\|_{L^{\infty}(B_{16}')})^{1-\alpha}
\left(\| \tu\|_{L^{\infty}(B_{16}^+)}+ \|u\|_{L^{\infty}(B_{16}')}\right)^{\alpha}
\|u\|_{L^{\infty}(B_{16}')}^{1-\alpha}\\
&+C(1+\|q\|_{L^{\infty}(B_{16}')})^{\frac{1-s}{1+s}}\left(\| \tu\|_{L^{\infty}(B_{16}^+)}+ \|u\|_{L^{\infty}(B_{16}')}\right)^{\frac{2s}{1+s}}
\|u\|_{L^{\infty}(B_{16}')}^{\frac{1-s}{1+s}}\\
&+ C\|q\|_{L^{\infty}(B_{16}')}^{\frac{1}{2}} \|u\|_{L^{\infty}(B_{16}')}\\
\leq&C (1+\|q\|_{L^{\infty}(B_{16}')})^{\frac{1-s}{1+s}}\left(
\left(\| \tu\|_{L^{\infty}(B_{16}^+)}+ \|u\|_{L^{\infty}(B_{16}')}\right)^{\alpha}
\|u\|_{L^{\infty}(B_{16}')}^{1-\alpha}\right.\\
&\left.+\left(\| \tu\|_{L^{\infty}(B_{16}^+)}+ \|u\|_{L^{\infty}(B_{16}')}\right)^{\frac{2s}{1+s}}
\|u\|_{L^{\infty}(B_{16}')}^{\frac{1-s}{1+s}}\right)\\
&+ C\|q\|_{L^{\infty}(B_{16}')}^{\frac{1}{2}} \|u\|_{L^{\infty}(B_{16}')}.
\end{split}
\end{align*}
Here we recall that 
\[
1-\alpha=\frac{1-s}{1+s}-\frac{1-s}{1+s}\alpha<\frac{1-s}{1+s}.
\]
\end{proof}

\begin{proof}[Proof of Proposition \ref{prop:boundary_bulk}]

\emph{Step 1: $L^{\infty}$ decay.}
We first prove that the $L^2$ bound in the statement of the proposition entails a similar $L^{\infty}$ bound. In the sequel, we denote by $\tilde{c},\tilde{C}$ general positive constants which may depend on $n$, $s$ and which are likely to change from line to line. Pick any $R\geq 1$ and $x_0 \in \R^n \times \{0\}$ with $|x_0|= 32 R$, \eqref{ubound} implies
\begin{align}
\label{eq:boundaryL2}
\|u\|_{L^2(B_{16 R}'(x_0))} \leq \tilde{C} e^{-\tilde{c}R^{\beta}} .
\end{align}
We next recall that by the $H^{s}(\R^n)$ boundedness of $u$ and the properties of the Caffarelli-Silvestre extension, we have $\| x_{n+1}^{\frac{1-2s}{2}} \nabla \tu \|_{L^2(\R^{n+1}_+)}\leq C$. Thus, the $L^2$ boundedness of $u$ and Poincar\'e's inequality then also yield a bound for $\|x_{n+1}^{\frac{1-2s}{2}} \tu\|_{L^2(\R^n \times (0,C_1))}$ for any $C_1 >0$.  
Combining the bound from Proposition \ref{prop:small_prop_bound}(a), the $L^2_{loc}(\R^{n+1}_+, x_{n+1}^{1-2s})$ boundedness of $\tilde{u}$, the fact that $\|q\|_{L^{\infty}} \leq 1$ and \eqref{eq:boundaryL2}, we infer that
\begin{align*}
\|x_{n+1}^{\frac{1-2s}{2}}\tilde u\|_{L^2(B_{c R}^+(x_0))} \leq \tilde{C} e^{-\tilde{c}R^{\beta}} .
\end{align*}
Here $c>0$ denotes the constant from Proposition \ref{prop:small_prop_bound}. Finally, invoking a translated and rescaled version of \eqref{eq:combi_up_Linf} then also entails the bound 
\begin{equation}\label{ubound2}
\|\tu\|_{L^{\infty}(B_{\frac{c^2R}{16}}^+(x_0))} \leq \tilde{C} e^{-\tilde{c}R^{\beta}} .
\end{equation}
This yields a bound for $\tu$ for $|x|\geq 32 R$. We may in particular use this for $R=10$. Hence, only on the compact set $B_{C}^+$ with $C=320$ an $L^{\infty}$ bound has not yet been obtained. This however follows by applying a rescaled version of \eqref{eq:LinfL2}. Thus, $\tu \in L^{\infty}(\R^{n+1}_+)$. The estimate in the bounded region and the quantitative estimate in the unbounded annuli can finally be combined to infer that for all $R>0$ and $x_0 \in \R^{n}\times \{0\}$, $|x_0|=32 R$ we have
\begin{equation*}
\|\tu\|_{L^{\infty}(B_{\frac{c^2R}{16}}^+(x_0))} \leq \tilde{C} e^{-\tilde{c}R^{\beta}} .
\end{equation*}
In particular, by choosing $R=\frac{16}{c} \tilde{R}$ and keeping $x_0 \in \R^{n}\times \{0\}$, $|x_0| = 32 R$, we also obtain the estimate 
\begin{equation}\label{ubound4}
\|\tu\|_{L^{\infty}(B_{c \tilde{R}}^+(x_0))} \leq \tilde{C} e^{-\bar{c}\tilde{R}^{\beta}} .
\end{equation}

\emph{Step 2: Conclusion.}
With the bounds from Step 1, Propositions \ref{prop:small_prop_int} and \ref{prop:small_prop_bound} at hand, the proof of Proposition \ref{prop:boundary_bulk} follows by a chain of balls argument. 
More precisely, for $x=(x', x_{n+1})$ with $|x|> 2$ there exists a value $R = 2^{k}$ for $k\in \N$ such that $x\in A_{R,2R}^+:=\{x\in \R^{n+1}_+: \ |x| \in (R,2R] \}$. This annulus can be covered by a finite union of balls and half balls:
\begin{align*}
A_{R,2R}^+ \subset \bigcup\limits_{j=1}^{m_1} B_{r_j}^+(x_j) \cup  \bigcup\limits_{k=1}^{m_2} B_{c R}^+((x_k',0)) 
\end{align*}
with the property that these balls form a chain, i.e. there is sufficient overlap between these to iterate the following estimates (c.f. Figure \ref{fig:annuli} for an illustration of this), where $c$ is the constant derived in Proposition \ref{prop:small_prop_bound}.
\begin{figure}
\includegraphics[width=0.6\textwidth]{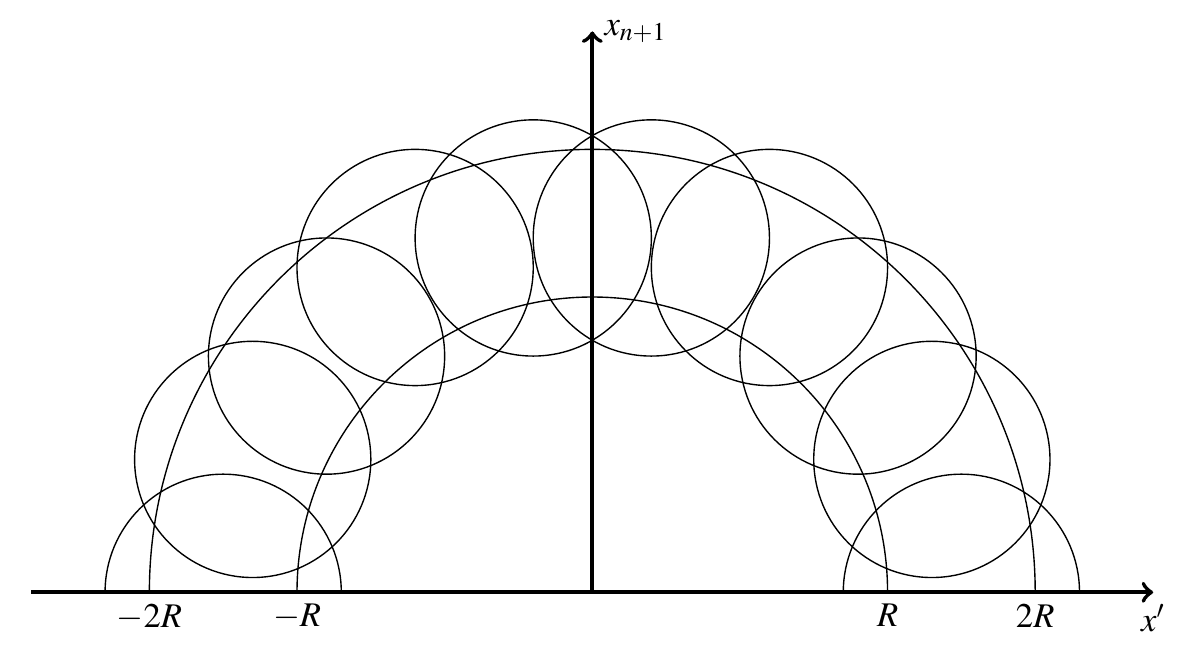}
\caption{The chain of balls from the proof of Proposition \ref{prop:boundary_bulk}.}
\label{fig:annuli}
\end{figure}
We explain this iteration more carefully.  Starting with a half ball $B_{cR}^+(\bar{x})$ with $\bar{x}=(x_1',0) \in \R^n \times \{0\}$ we invoke \eqref{ubound4} to infer that
\begin{align*}
\| \tu \|_{L^{\infty}(B_{c R}^+(\bar{x}))}\leq C e^{-aR^{\beta}},
\end{align*}
where $a=a(n,s)>0$. 

We now begin with the propagation of decay estimates into $\R^{n+1}_+$ along a chain of balls. We first choose a ball $B_{r_1}^+(x_1) \subset B_{c R}^+(\bar{x})$ with $r_1 = \frac{c R}{5}$ in such a way that a large part of $B_{r_1}^+(x_1) \subset A_{R,2R}^+$, $B_{4r_1}^+ \subset \R^{n+1}_+$  and such that $|B_{2 r_1}^+(x_1) \Delta B_{c R}^+(\bar{x})| \geq c_0 R$ for some constant $c_0=c_0(n,s)>0$. In $B_{4r_1}^+(x_1)$ we apply the three balls inequality of Proposition \ref{prop:small_prop_int}. Therefore, in combination with the $L^{\infty}$ estimates from step 1, we obtain
\begin{align*}
\|\tu\|_{L^{\infty}(B_{2 r_1}^+(x_1))} 
&\leq C \|\tu\|_{L^{\infty}(B_{4 r_1}^+(x_1))}^{\alpha} \|\tu\|_{L^{\infty}(B_{r_1}^+(x_1))}^{1-\alpha}\\
& \leq C \|\tu\|_{L^{\infty}(B_{c R}^+(\bar x))}^{1-\alpha}\\
& \leq C e^{-a(1-\alpha) R^{\beta}}. 
\end{align*}
We iterate this along our chain of balls $B_{r_j}^+(x_j)$ with $r_j=\frac{cR}{5}$, which allows us to eventually cover $A_{R,2R}^+$ with a finite number of balls. In particular, it yields the decay estimate
\begin{align*}
\|\tu\|_{L^{\infty}(A_{R,2R}^+)} 
& \leq C e^{-\tilde a R^{\beta}}
\end{align*}
for some positive constants $C(m_1,m_2,n,s), \tilde a(m_1,m_2,n,s)$. Finally, the analogous bounds also follow in arbitrary other dyadic annuli centered at zero by rescaling the previous estimate (the number of balls in the chain of balls estimate stays constant, since the size of the balls is also rescaled).
\end{proof}

\subsection{An interpolation inequality}\label{sec:interpol}

We need another bulk-boundary interpolation estimate, which will play a relevant role in our Carleman inequalities in the next section. Although this was already introduced in \cite{Rue15}, we reprove it here for self-containedness.

\begin{prop}
\label{prop:interpol}
Let $s\in(0,1)$ and $u:S^{n}_+ \rightarrow \R$ with $u\in H^{1}(S^n_+,\theta_n^{1-2s} )$. Then there exists a constant $C=C(s,n)>0$ such that for all $\tau>1$
\begin{align*}
\|u\|_{L^2(S^{n-1})} \leq C (\tau^{1-s} \|\theta^{\frac{1-2s}{2}}_n u\|_{L^2(S^{n}_+)} + \tau^{-s} \|\theta_n^{\frac{1-2s}{2}} \nabla_{S^n} u\|_{L^2(S^n_+)} ).
\end{align*}
\end{prop}

\begin{proof}
We argue in two steps.\\

\emph{Step 1: Derivation of a whole space estimate.} Let $w \in H^{s}(\R^n)\cap H^{1}(\R^{n+1}_+, x_{n+1}^{1-2s})$. Then trace estimates in the space $H^{1}(\R^{n+1}_+, x_{n+1}^{1-2s})$ (c.f. for instance Lemma 4.4 in \cite{RS17}) imply
\begin{equation}
\label{eq:trace}
\|w\|_{L^2(\R^n)}\leq C(\|x_{n+1}^{\frac{1-2s}{2}} w\|_{L^2(\R^{n+1}_+)} 
+ \|x_{n+1}^{\frac{1-2s}{2}} \nabla w \|_{L^2(\R^{n+1}_+)}).
\end{equation}
Starting from \eqref{eq:trace}, scaling $x$ by $\tau^{-1}$ with $\tau>0$ (i.e., $x\to \tau^{-1}x$), we then obtain
\begin{align}
\label{eq:interpol_add}
\|w\|_{L^2(\R^n)}
\leq C( \tau^{1-s} \|x_{n+1}^{\frac{1-2s}{2}} w\|_{L^2(\R^{n+1}_+)} + \tau^{-s} \|x_{n+1}^{\frac{1-2s}{2}} \nabla w\|_{L^2(\R^{n+1}_+)}).
\end{align}

\emph{Step 2: Conclusion.}
Considering $u \in H^1(S^{n}_+, \theta_n^{1-2s})$, we first extend this function zero homogeneously into a neighbourhood of $S^n_+$ and multiply it by a cut-off function, i.e. we define $w(x):= \eta(x) u(\frac{x}{|x|})$, where $\eta(x) = 1$ if $|x| \in (1/2,3/2)$ and $\eta(x) = 0$ if $|x|\in (0,1/4)\cup (2,\infty)$. The resulting compactly supported function still satisfies $w\in H^{1}(\R^{n+1}_+, x_{n+1}^{1-2s})$ and further has the property that
\begin{align*}
\|x_{n+1}^{\frac{1-2s}{2}}\nabla w\|_{L^2(\R^{n+1}_+)}
&\leq C( \|\theta_n^{\frac{1-2s}{2}} \nabla_{S^n} u\|_{L^2(S^n_+)} + 
\|\theta_n^{\frac{1-2s}{2}} u\|_{L^2(S^n_+)}),\\
\|x_{n+1}^{\frac{1-2s}{2}} w\|_{L^2(\R^{n+1}_+)}
&\leq C \|\theta_n^{\frac{1-2s}{2}} u\|_{L^2(S^n_+)},\\
\|u\|_{L^2(S^{n-1})} &\leq C \|w\|_{L^2(\R^n)}.
\end{align*}
Inserting these into \eqref{eq:interpol_add} and choosing $\tau \geq \tau_0 > 1$ (for some uniform $\tau_0$) then implies
\begin{align*}
\|u\|_{L^2(S^{n-1})} 
& \leq C \|w\|_{L^2(\R^{n})}
\leq C(\tau^{1-s} \|x_{n+1}^{\frac{1-2s}{2}} w\|_{L^2(\R^{n+1}_+)} + \tau^{-s}\|x_{n+1}^{\frac{1-2s}{2}} \nabla w\|_{L^2(\R^{n+1})})\\
&\leq C(\tau^{1-s}(1+\tau^{-1}) \|\theta_n^{\frac{1-2s}{2}} u \|_{L^2(S^n_+)} + \tau^{-s} \|\theta_n^{\frac{1-2s}{2}} u\|_{L^2(S^n_+)} ).
\end{align*}
Using that $\tau \geq 1$ then implies the claimed estimate.
\end{proof}

\section{Carleman Inequalities}
\label{eq:Carl}

In the following two sections, we prove the Carleman estimates, which provide the main tools in deriving the decay estimates of Theorems \ref{thm:Landis_Diff} and \ref{thm:Landis_No_Diff}. 

\subsection{A Carleman inequality under differentiability assumptions}
\label{sec:Carl_Diff}

We begin with an estimate in the setting of differentiable potentials. Here we ``include" the potential into the estimate, which allows us to obtain better boundary contributions. This is however at the expense of requiring radial differentiability properties on the potential. It corresponds to a similar argument from \cite{M91} for $s=1$ in the case of differentiable potentials.

\begin{thm}
\label{thm:Carl_Diff}
Let $s\in (0,1)$ and let $\tilde{u}\in H^{1}(\R^{n+1}_+, x_{n+1}^{1-2s})$ with $\supp(\tilde{u}) \subset \R^{n+1}_+ \setminus B_1^+$ be a solution to
\begin{align*}
\nabla \cdot x_{n+1}^{1-2s} \nabla \tilde{u} & = f \mbox{ in } \R^{n+1}_+,\\
\lim\limits_{x_{n+1} \rightarrow 0} x_{n+1}^{1-2s} \p_{n+1} \tilde{u }& = V \tilde{u} \mbox{ on } \R^n \times \{0\},
\end{align*}
where $f\in L^{2}(\R^{n+1}_+, x_{n+1}^{2s-1})$ with compact support in $\R^{n+1}_+$, $V \in C^1_r(\R^{n}\times \{0\})$, i.e. $V\in C^{0}(\R^n \times \{0\})$ and $x\cdot \nabla V$ exists.
Let further $\phi(x) = |x|^{\alpha}$ for $\alpha \geq 1$. Then there exists a constant $C>1$ such that for all $\tau \geq \tau_0>1$ it holds
\begin{align}
\label{eq:Carl_cart_Diff}
\begin{split}
&\tau^{3} \|e^{\tau \phi} |x|^{\frac{3 \alpha}{2}-1} x_{n+1}^{\frac{1-2s}{2}} \tilde{u}\|_{L^2(\R^{n+1}_+)}^2 + \tau \|e^{\tau \phi} |x|^{\frac{\alpha}{2}}x_{n+1}^{\frac{1-2s}{2}}\nabla \tilde{u}\|_{L^2(\R^{n+1}_+)}^2\\
&\leq C \left(\| e^{\tau \phi} x_{n+1}^{\frac{2s-1}{2}} |x| f\|_{L^2(\R^{n+1}_+)}^2  \right. \\
& \quad \left. + \tau \| e^{\tau \phi} |x|^{\frac{\alpha}{2}}  |(x\cdot \nabla V)|^{\frac{1}{2}}\tilde{u}\|_{L^2(\R^n \times \{0\})}^2 
+ \tau \|e^{\tau \phi} |x|^{\frac{\alpha}{2}}|V|^{\frac{1}{2}} \tilde{u}\|_{L^2(\R^n \times \{0\})}^2 
\right).
\end{split}
\end{align}
\end{thm}

\begin{proof}
We first pass to conformal polar coordinates. To this end, we define $x= e^{t} \theta$ with $t\in \R$, $\theta \in S^n_+$, set $\bar{u}(t,\theta) = e^{\frac{(n-2s)}{2}t}\tilde u(e^t \theta)$ and multiply the resulting equation for $\bar{u}$ with $e^{\frac{(n-2s)}{2}t}$. This then leads to the equation
\begin{align*}
%\label{eq:elliptic_conf}
\begin{split}
\left( \theta_n^{1-2s} \p_t^2 + \nabla_{S^n}\cdot \theta_n^{1-2s} \nabla_{S^n} - \theta_n^{1-2s} \frac{(n-2s)^2}{4} \right) \bar{u} & = \tilde{f} \mbox{ in } S^{n}_+ \times \R,\\
\lim\limits_{\theta_n \rightarrow 0} \theta_n^{1-2s} \p_\nu \bar{u} &= \tilde{V} \bar{u} \mbox{ on } S^{n-1} \times \R.
\end{split}
\end{align*}
Here $\tilde{f}(t,\theta) := e^{\frac{(n+2+2s)}{2}t} f(e^t \theta), \tilde{V}(t,\theta) := e^{2st} V(e^t \theta) $ and $\p_\nu=\nu\cdot\nabla_{S^n}$ with $\nu=(0,\cdots,0,1)$. Next, setting $v= e^{\tau \varphi(t)} \bar{u}$ and $\bar f=e^{\tau \varphi(t)}\tilde f$ (with $\varphi(t)=\phi(e^t \theta) = e^{\alpha t}$), we seek to prove the following estimate, which (after returning to Cartesian coordinates) implies \eqref{eq:Carl_cart_Diff}:
\begin{align}
\label{eq:Carl_conf_Diff}
\begin{split}
&\tau^{3}\|\varphi' |\varphi''|^{\frac{1}{2}} \theta_n^{\frac{1-2s}{2}}  v\|_{L^2(S^{n}_+\times \R)}^2 
 + \tau \||\varphi''|^{\frac{1}{2}} \theta_n^{\frac{1-2s}{2}} \nabla_{S^{n}} v\|_{L^2(S^{n}_+\times \R)}^2
 + \tau \||\varphi''|^{\frac{1}{2}} \theta_n^{\frac{1-2s}{2}} \p_t v\|_{L^2(S^{n}_+ \times \R)}^2 
 \\
&\leq C \left( \|\theta_n^{\frac{2s-1}{2}} \bar{f}\|_{L^2(S^n_+ \times \R)}^2 + \tau \| |\varphi'|^{\frac{1}{2}} |\p_t \tilde{V}|^{\frac{1}{2}} v \|_{L^2(S^{n-1}\times \R)}^2 + \tau \| |\varphi''|^{\frac{1}{2}} |\tilde{V}|^{\frac{1}{2}} v \|_{L^2(S^{n-1}\times \R)}^2  \right).
\end{split}
\end{align}

In order to infer this, we consider the function $\overline{v}:=\theta_n^{\frac{1-2s}{2}} v$, which satisfies the equation
\begin{align*}
L_\varphi\bar v:=\left(  \p_t^2 + \tilde{\D}_{S^n} - \frac{(n-2s)^2}{4}  + \tau^2 |\varphi'|^2 - 2\tau \varphi' \p_t - \tau \varphi''  \right) \bar{v} & = \theta^{\frac{2s-1}{2}}\bar{f} \mbox{ in } S^{n}_+ \times \R,\\
\lim\limits_{\theta_n \rightarrow 0} \theta_n^{1-2s} \p_\nu \theta_n^{\frac{2s-1}{2}} \bar{v} & = \tilde{V} \theta_n^{\frac{2s-1}{2}} \bar{v} \mbox{ on } S^{n-1} \times \R,
\end{align*}
where $\tilde{\D}_{S^n}:= \theta_n^{\frac{2s-1}{2}} \nabla_{S^n} \cdot \theta_n^{1-2s} \nabla_{S^n} \theta_n^{\frac{2s-1}{2}}$.
We split the bulk operator into its (formally) symmetric and antisymmetric parts:
\begin{align*}
S & = \p_t^2 + \tilde{\D}_{S^n} + \tau^2 |\varphi'|^2 - \frac{(n-2s)^2}{4} ,\\
A & = -2 \tau \varphi' \p_t - \tau \varphi'' .
\end{align*}
Then, using that $S,A$ are only symmetric and antisymmetric up to boundary contributions, we obtain
\begin{align}
\label{eq:L2expand}
\|L_{\varphi} \bar{v}\|^2 = 
\|S \bar{v}\|^2 + \|A \bar{v}\|^2 + ([S,A]\bar{v},\bar{v}) +  4 \tau (\tilde{V} \theta_n^{\frac{2s-1}{2}} \bar{v}, \varphi' \theta_n^{\frac{2s-1}{2}} \p_t \bar{v})_{0}+2 \tau (\tilde{V} \theta_n^{\frac{2s-1}{2}} \bar{v}, \varphi'' \theta_n^{\frac{2s-1}{2}} \bar{v})_{0}.
\end{align}
For abbreviation, we have here set
\begin{align}
\label{eq:notation}
\begin{split}
\|\cdot\| & = \|\cdot \|_{L^2(S^n_+ \times \R)}, \
(\cdot, \cdot):=(\cdot, \cdot)_{L^2(S^{n}_+ \times \R)},\\
\|\cdot \|_0 &:= \| \cdot \|_{L^2(S^{n-1} \times \R)},\
(\cdot, \cdot)_0 := (\cdot, \cdot)_{L^2(S^{n-1}\times \R)}.
\end{split}
\end{align}
In the sequel, we will use this notation frequently. The bulk terms are bounded as usual. More precisely, the commutator reads
\begin{align}
\label{eq:bulk}
\begin{split}
([S,A]\bar{v},\bar{v})
&= 4 \tau^3( \varphi'' |\varphi'|^2 \bar{v},\bar{v}) - 4\tau (\varphi'' \p_t^2 \bar{v},\bar{v}) - 4 \tau(\varphi''' \p_t \bar{v},\bar{v}) - \tau (\varphi'''' \bar{v},\bar{v} )\\
&= 4 \tau^3( \varphi'' |\varphi'|^2 \bar{v},\bar{v}) + 4\tau (\varphi'' \p_t \bar{v}, \p_t \bar{v}) - \tau (\varphi'''' \bar{v},\bar{v} )\\
& \geq  2 \tau^3( \varphi'' |\varphi'|^2 \bar{v},\bar{v}) + 4\tau (\varphi'' \p_t \bar{v}, \p_t \bar{v}) .
\end{split}
\end{align}
In the last line, we used the growth of $\varphi$ to absorb the last term into the first term on the right hand side (for a sufficiently large choice of $\tau_0>0$). This yields the $L^2$ and the radial part of the gradient $L^2$ bulk contributions. 

In order to also obtain the full gradient estimate, we exploit the symmetric part of the operator. By an integration by parts argument we infer for some small constant $c_0$
\begin{align}
\label{eq:sph_grad}
\begin{split}
c_0 \tau \||\varphi'|^{\frac{1}{2}} \theta_n^{\frac{1-2s}{2}} \nabla_{S^n} \theta_n^{\frac{2s-1}{2}} \bar{v}\|^2
&\leq c_0 \tau |(S \bar{v}, \varphi' \bar{v})|
+ c_0 \tau \||\varphi''|^{\frac{1}{2}} \p_t \bar{v}\|^2
+ c_0\tau^3\||\varphi'|^{\frac{3}{2}}\bar v\|^2\\
& \quad + c_0 \tau \left( \frac{n-2s}{2} \right)^{2} \||\varphi'|^{\frac{1}{2}} \bar{v}\|^2 + c_0\tau |(\theta_n^{1-2s}\p_{\nu} \theta_n^{\frac{2s-1}{2}}\bar{v}, \varphi'\theta_n^{\frac{2s-1}{2}} \bar{v})_0|\\
&\leq c_0 \tau |(S \bar{v}, \varphi'\bar{v})|
+ c_0 \tau \||\varphi''|^{\frac{1}{2}} \p_t \bar{v}\|^2
+ c_0 \tau^3 \|\varphi'|\varphi''|^{\frac{1}{2}} \bar{v}\|^2\\
& \quad + c_0 \tau \left( \frac{n-2s}{2} \right)^{2} \||\varphi'|^{\frac{1}{2}} \bar{v}\|^2 + c_0 \tau |(\tilde{V} \theta_n^{\frac{2s-1}{2}} \bar{v},\varphi' \theta_n^{\frac{2s-1}{2}}  \bar{v})_0|.
\end{split}
\end{align}
Combining \eqref{eq:L2expand}, \eqref{eq:bulk}, \eqref{eq:sph_grad} yields 
\begin{align*}
&c_0 \tau \||\varphi'|^{\frac{1}{2}} \theta_n^{\frac{1-2s}{2}} \nabla_{S^n} \theta_n^{\frac{2s-1}{2}} \bar{v}\|^2
+ 2 \tau^3 \|\varphi'|\varphi''|^{\frac{1}{2}} \bar{v}\|^2
+ 4 \tau \||\varphi''|^{\frac{1}{2}} \p_t \bar{v}\|^2
+ \|S \bar{v}\|^2 + \|A \bar{v}\|^2\\
&\leq c_0 \tau |(S \bar{v}, \varphi'\bar{v})| + c_0 \tau \||\varphi''|^{\frac{1}{2}} \p_t \bar{v}\|^2
+ c_0 \tau^3 \|\varphi'|\varphi''|^{\frac{1}{2}} \bar{v}\|^2\\
& \quad + c_0 \tau \left( \frac{n-2s}{2} \right)^{2} \||\varphi'|^{\frac{1}{2}} \bar{v}\|^2 + c_0 \tau |(\tilde{V} \theta_n^{\frac{2s-1}{2}} \bar{v}, \varphi'\theta_n^{\frac{2s-1}{2}}  \bar{v})_0|\\
& \quad + \|L_{\varphi}\bar{v}\|^2 + 4 \tau |(\tilde{V} \theta_n^{\frac{2s-1}{2}}\bar{v}, \varphi' \theta_n^{\frac{2s-1}{2}} \p_t \bar{v})_0|+ 2 \tau |(\tilde{V} \theta_n^{\frac{2s-1}{2}} \bar{v}, \varphi'' \theta_n^{\frac{2s-1}{2}} \bar{v})_{0}|.
\end{align*} 
Noting that the bulk terms on the right hand side of this estimate can be absorbed into the left hand side, if $c_0$ is chosen sufficiently small and $\tau_0 \geq 1$ sufficiently large, entails
\begin{align}
\label{eq:Carl_bulk_ok}
\begin{split}
&c_0 \tau \||\varphi'|^{\frac{1}{2}} \theta_n^{\frac{1-2s}{2}} \nabla_{S^n} \theta_n^{\frac{2s-1}{2}} \bar{v}\|^2
+  \tau^3 \|\varphi'|\varphi''|^{\frac{1}{2}} \bar{v}\|^2
+  \tau \||\varphi''|^{\frac{1}{2}} \p_t \bar{v}\|^2\\
&\leq \|L_\varphi \bar{v}\|^2 + 4 \tau |(\tilde{V} \theta_n^{\frac{2s-1}{2}}\bar{v}, \varphi' \theta_n^{\frac{2s-1}{2}} \p_t \bar{v})_0|  + 3\tau |(\tilde{V} \theta_n^{\frac{2s-1}{2}} \bar{v}, \varphi''\theta_n^{\frac{2s-1}{2}}  \bar{v})_0|.
\end{split}
\end{align} 
It thus remains to control the boundary contribution involving the $t$-derivative (the other one is already of the desired form, c.f. \eqref{eq:Carl_conf_Diff}).
It is controlled by integrating by parts in $t$ which leads to
\begin{align}
\label{eq:boundary}
\begin{split}
\tau \left| (\tilde{V} \theta_n^{\frac{2s-1}{2}} \bar{v}, \varphi' \theta_n^{\frac{2s-1}{2}} \p_t \bar{v})_0 \right|
& \leq \tau \left| ((\p_t \tilde{V}) \theta_n^{\frac{2s-1}{2}} \bar{v}, \varphi' \theta_n^{\frac{2s-1}{2}} \bar{v})_0 \right| \\
& \quad + \tau \left| ( \tilde{V} \theta_n^{\frac{2s-1}{2}} \bar{v}, \varphi'' \theta_n^{\frac{2s-1}{2}} \bar{v})_0 \right|.
\end{split}
\end{align}
Returning from $\bar{v}$ to $v$ shows that the boundary contributions on the right hand side of \eqref{eq:boundary} are exactly controlled by the boundary contributions in \eqref{eq:Carl_conf_Diff}. 
Thus, finally, combining \eqref{eq:bulk} with \eqref{eq:boundary} yields \eqref{eq:Carl_conf_Diff}, which concludes the proof for Theorem \ref{thm:Carl_Diff}. % spherical grad through equation if we want it, but maybe not even needed.
\end{proof}

\subsection{A Carleman inequality without differentiability assumptions}
\label{sec:Carl_No_Diff}

In this section we prove a similar Carleman estimate as in the previous section. However, in contrast to the previous estimate, we now do not presuppose any differentiability properties on the potential $V$. As in the classical case $s=1$ this implies that we can no longer treat the potential $V$ as part of the operator, but instead have to deal with it perturbatively. This however does no longer allow us to distinguish between the complex (system) and the real valued (scalar) case. Hence, we can only derive weaker estimates. 

\begin{thm}
\label{thm:Carl_No_Diff}
Let $s\in (0,1)$ and let $\tilde{u}\in H^{1}(\R^{n+1}_+, x_{n+1}^{1-2s})$ with $\supp(\tilde{u}) \subset B_{R}^+ \setminus B_1^+$ for some constant $R>1$ be a solution to
\begin{align*}
\nabla \cdot x_{n+1}^{1-2s} \nabla \tilde{u} & = f\; \mbox{ in }\; \R^{n+1}_+,\\
\lim\limits_{x_{n+1} \rightarrow 0} x_{n+1}^{1-2s} \p_{n+1} \tilde{u }& = V \tilde{u}\; \mbox{ on }\; \R^n \times \{0\},
\end{align*}
where $f\in L^2(\R^{n+1}_+, x_{n+1}^{2s-1})$ with compact support in $\R^{n+1}_+$ and $V \in L^{\infty}(\R^n)$. 
Let further $\phi(x) = |x|^{\alpha}$ for some $\alpha \ge 1$.

Then there exist constants $C, \tau_0 >1$ such that for all $\tau \geq \tau_0$ it holds
\begin{align*}
&\tau^{3} \|e^{\tau \phi}|x|^{\frac{3\alpha}{2}-1} x_{n+1}^{\frac{1-2s}{2}} \tilde{u}\|_{L^2(\R^{n+1}_+)}^2 + \tau \|e^{\tau \phi} |x|^{\frac{\alpha}{2}}x_{n+1}^{\frac{1-2s}{2}} \nabla \tilde{u}\|_{L^2(\R^{n+1}_+)}^2\\
&\leq C \left(\| e^{\tau \phi} x_{n+1}^{\frac{2s-1}{2}} |x| f\|_{L^2(\R^{n+1}_+)}^2 +  \tau^{2-2s} \|e^{\tau \phi}  V |x|^{(1-\alpha)s} \tilde{u}\|_{L^2(\R^n \times \{0\})}^2 
\right).
\end{align*}
\end{thm}

\begin{proof}
As in \cite{GRSU18}, we deduce the estimate by a splitting argument. To this end, we first pass to conformal polar coordinates. \\

\emph{Step 1: Conformal coordinates and set-up.} As in the proof of Theorem \ref{thm:Carl_Diff}, we first pass to conformal polar coordinates. With the notation from there, we obtain
\begin{align}
\label{eq:elliptic_conf}
\begin{split}
\left( \theta_n^{1-2s} \p_t^2 + \nabla_{S^n}\cdot \theta_n^{1-2s} \nabla_{S^n} - \theta_n^{1-2s} \frac{(n-2s)^2}{4} \right) \bar{u} & = \tilde{f} \mbox{ in } S^{n}_+ \times \R,\\
\lim\limits_{\theta_n \rightarrow 0} \theta_n^{1-2s} \p_{\nu} \bar{u} = \tilde{V} \bar{u} \mbox{ on } S^{n-1} \times \R.
\end{split}
\end{align}

In order to deduce the desired exponential estimates for this problem, we split the function $\bar{u}$ into two parts $\bar{u}= u_1 + u_2 $. Here $u_1$ is a solution to
\begin{align}
\label{eq:elliptic}
\begin{split}
\left( \theta_n^{1-2s} \p_t^2 + \nabla_{S^n}\cdot \theta_n^{1-2s} \nabla_{S^n} - \theta_n^{1-2s} \frac{(n-2s)^2}{4} - K^2 \tau^2 |\varphi'|^2 \theta_n^{1-2s}  \right) u_1 & = \tilde{f} \mbox{ in } S^{n}_+ \times \R,\\
\lim\limits_{\theta_n \rightarrow 0} \theta_n^{1-2s} \p_{\nu} u_1 = \tilde{V} \bar{u} \mbox{ on } S^{n-1} \times \R.
\end{split}
\end{align}
The constant $K \in \R$ is sufficiently large and is to be determined more precisely later; further $\varphi(t) := \phi(e^t)$. 
We remark that by the Lax-Milgram theorem in $H^{1}(S^n_+ \times \R, \theta_n^{1-2s})$ a unique energy solution to this problem exists. Further, by arguments similar as in the Appendix of  \cite{GRSU18} this function is rapidly decaying at infinity. The equation for $u_2$ follows correspondingly. In order to infer the desired Carleman estimate, we combine elliptic estimates for $u_1$ with the usual commutator estimates for $u_2$. 
We deduce these estimates separately and begin by discussing the elliptic bounds for $u_1$. \\

\emph{Step 2: Elliptic estimates for $u_1$.}
Using the same notational convention as in \eqref{eq:notation} and testing the weak form of the equation \eqref{eq:elliptic} by $\tau^2 e^{2\tau \varphi}|\varphi''|^2 u_1$, we obtain
\begin{align}
\label{eq:test_ell}
\begin{split}
& \tau^2 ( e^{2\tau \varphi}|\varphi''|^2 \theta_n^{1-2s} \p_t  u_1, \p_t u_1 ) 
+ \tau^2 (e^{2 \tau \varphi} |\varphi''|^2 \theta_n^{1-2s} \nabla_{S^n} u_1, \nabla_{S^n} u_1)\\
& + \tau^2 \frac{(n-2s)^2}{4}(e^{2 \tau \varphi} |\varphi''|^2 \theta_n^{1-2s} u_1, u_1) + K^2 \tau^4 (e^{2\tau \varphi} |\varphi'|^2 |\varphi''|^2 \theta_n^{1-2s}u_1, u_1)\\
& = - \tau^2 (e^{2\tau \varphi} \tilde{f}, |\varphi''|^2 u_1 ) 
+ \tau^2 (e^{2 \tau \phi}\tilde{V}\bar{u},|\varphi''|^2 u_1)_0\\
& \quad-2\tau^3 (e^{2\tau \varphi}|\varphi''|^2 \varphi' \theta_n^{1-2s}\p_t u_1, u_1) -2\tau^2(e^{2 \tau \varphi} \varphi''\varphi'''u_1, \theta_n^{1-2s} \p_t u_1).
\end{split}
\end{align}
Choosing $K\geq 1$ sufficiently large and applying Young's inequality, it is thus possible to absorb unsigned contributions from the right hand side into the left hand side. This results in
\begin{align}
\label{eq:test_ell1}
\begin{split}
& \tau^2 \| e^{\tau \varphi} \varphi'' \theta_n^{\frac{1-2s}{2}} \p_t u_1 \|^2
+ \tau^2 \|e^{\tau \varphi} \varphi'' \theta_n^{\frac{1-2s}{2}} \nabla_{S^n}u_1 \|^2
+ \frac{1}{2} K^2 \tau^4 \|e^{\tau \varphi} \varphi'\varphi''  \theta_n^{\frac{1-2s}{2}}u_1\|^2\\
&\leq C(\|e^{\tau \varphi} \theta_n^{\frac{2s-1}{2}} \tilde{f}\|^2 + \tau^{2-2s}\| e^{\tau \varphi}\varphi'' \tilde{V} e^{-\alpha s  t} \overline{u}\|^2_0) + \epsilon \tau^{2+2s} \| e^{\tau \varphi} \varphi'' e^{\alpha s t} u_1 \|^2_0.
\end{split}
\end{align}
Finally, in order to conclude the discussion on the function $u_1$, we bound the boundary contribution $ \tau^{2+2s} \| e^{2\tau \varphi} \varphi'' e^{\alpha s t} u_1 \|_0^2$ by means of the bulk-boundary interpolation estimate from Proposition \ref{prop:interpol}. Recalling that $\varphi(t) = e^{\alpha t}$ and treating the non-spherical variables as constants, we obtain
\begin{align*}
|\varphi''|^2 e^{2 \alpha st} \int\limits_{S^{n-1}} u_1^2 d\theta 
& \leq C (\tilde{\tau}^{2-2s}  |\varphi''|^2 e^{2 \alpha st} \int\limits_{S^{n}_+} \theta_n^{1-2s} u_1^2 d\theta  + \tilde{\tau}^{-2s}  |\varphi''|^2 e^{2 \alpha st} \int\limits_{S^{n}_+} \theta_n^{1-2s} |\nabla_{S^{n}} u_1|^2 d\theta  ).
\end{align*}
Choosing $\tilde{\tau} = e^{\alpha t} \tau$ (such that both the $L^2$ and the gradient contribution obtain radial weights which match the elliptic bulk estimates from \eqref{eq:test_ell1}), we obtain
\begin{align*}
|\varphi''|^2 e^{2 \alpha st} \int\limits_{S^{n-1}} u_1^2 d\theta 
& \leq C (\tau ^{2-2s} |\varphi''|^2 e^{2 \alpha t} \int\limits_{S^{n}_+} \theta_n^{1-2s} u_1^2 d\theta  + \tau^{-2s} |\varphi''|^2 \int\limits_{S^{n}_+} \theta_n^{1-2s} |\nabla_{S^{n}} u_1|^2 d\theta  ).
\end{align*}
Multiplying with $e^{2\tau \varphi}$, using that $\varphi' = \alpha e^{\alpha t}$ and integrating in the radial direction thus implies
\begin{align}
\label{eq:ell_bound_b}
\begin{split}
\tau^{2+2s} \| e^{\tau \varphi} |\varphi''| e^{\alpha s t} u_1 \|^2_0
&\leq C (\tau^{4} \|e^{\tau \varphi}  \theta_n^{\frac{1-2s}{2}}e^{\alpha t}\varphi'' u_1 \|^2 + \tau^2 \|e^{\tau \varphi}  \theta_n^{\frac{1-2s}{2}} \varphi'' \nabla_{S^n} u_1 \|^2)\\
&\leq C (\tau^{4} \|e^{\tau \varphi}  \theta_n^{\frac{1-2s}{2}}\varphi'\varphi'' u_1 \|^2 + \tau^2 \|e^{\tau \varphi}  \theta_n^{\frac{1-2s}{2}} \varphi'' \nabla_{S^n} u_1 \|^2).
\end{split}
\end{align}
Therefore, we may absorb the last boundary term in \eqref{eq:test_ell1} into the left hand side of \eqref{eq:test_ell1} and are left with
\begin{align}
\label{eq:test_ell2}
\begin{split}
& \tau^2 \| e^{\tau \varphi} \varphi'' \theta_n^{\frac{1-2s}{2}} \p_t u_1 \|^2
+ \tau^2 \|e^{\tau \varphi} \varphi'' \theta_n^{\frac{1-2s}{2}} \nabla_{S^n}u_1 \|^2
+ K^2 \tau^4 \|e^{\tau \varphi}\varphi'\varphi''  \theta_n^{\frac{1-2s}{2}}u_1\|^2\\
&\leq C(\|e^{\tau \varphi} \theta_n^{\frac{2s-1}{2}} \tilde{f}\|^2 + \tau^{2-2s}\| e^{\tau \varphi} \varphi'' \tilde{V} e^{-\alpha s  t} \overline{u}\|^2_0).
\end{split}
\end{align}

\emph{Step 3: Commutator estimates for $u_2$.}
Next we deal with the estimate for the function $u_2$, which follows from a commutator estimate similarly as in the proof of Theorem \ref{thm:Carl_Diff}. Indeed, $u_2$ satisfies the equation
\begin{align}
\label{eq:comm_u2}
\begin{split}
\left( \theta_n^{1-2s} \p_t^2 + \nabla_{S^n}\cdot \theta_n^{1-2s} \nabla_{S^n} - \theta_n^{1-2s} \frac{(n-2s)^2}{4}  \right) u_2 & =-  K^2 \tau^2 |\varphi'|^2 \theta_n^{1-2s}u_1 \mbox{ in } S^{n}_+ \times \R,\\
\lim\limits_{\theta_n \rightarrow 0} \theta_n^{1-2s} \p_{\nu} u_1 = 0 \mbox{ on } S^{n-1} \times \R.
\end{split}
\end{align}
In order to deduce the desired exponential estimates from this, we carry out a similar commutator argument as in the proof of Theorem \ref{thm:Carl_Diff}. In this procedure, we note that now the boundary terms drop out due to the vanishing Neumann condition. With this observation and exactly the same commutator bounds as in the proof of Theorem \ref{thm:Carl_Diff}, we therefore obtain the estimate
\begin{align}
\label{eq:comm_u2_2}
\begin{split}
&\tau^3\|e^{\tau \varphi} \varphi'|\varphi''|^{\frac{1}{2}} \theta_n^{\frac{1-2s}{2}} u_2\|^2
+ \tau \|e^{\tau \varphi} |\varphi''|^{\frac{1}{2}}\theta_n^{\frac{1-2s}{2}} \p_t u_2\|^2
+ \tau \|e^{\tau \varphi} |\varphi''|^{\frac{1}{2}} 
\theta_n^{\frac{1-2s}{2}} \nabla_{S^n} u_2 \|^2\\
&\leq C K^4 \tau^4 \| e^{\tau \varphi} |\varphi'|^2\theta_n^{\frac{1-2s}{2}} u_1\|^2.
\end{split}
\end{align}

\emph{Step 4: Conclusion.}
Finally, we combine the estimates from \eqref{eq:test_ell2} and \eqref{eq:comm_u2_2}. By the triangle inequality, this gives
\begin{align*}
&\tau^3\|e^{\tau \varphi} \varphi'|\varphi''|^{\frac{1}{2}} \theta_n^{\frac{1-2s}{2}} \bar{u}\|^2
+ \tau \|e^{\tau \varphi} |\varphi''|^{\frac{1}{2}}\theta_n^{\frac{1-2s}{2}} \p_t \bar{u} \|^2
+ \tau \|e^{\tau \varphi} |\varphi''|^{\frac{1}{2}} 
\theta_n^{\frac{1-2s}{2}} \nabla_{S^n} \bar{u} \|^2\\
& \leq  \tau^3\|e^{\tau \varphi} \varphi'|\varphi''|^{\frac{1}{2}} \theta_n^{\frac{1-2s}{2}} u_1 \|^2
+ \tau \|e^{\tau \varphi} |\varphi''|^{\frac{1}{2}}\theta_n^{\frac{1-2s}{2}} \p_t u_1 \|^2
+ \tau \|e^{\tau \varphi} |\varphi''|^{\frac{1}{2}} 
\theta_n^{\frac{1-2s}{2}} \nabla_{S^n} u_1 \|^2\\
& \quad + \tau^3\|e^{\tau \varphi} \varphi'|\varphi''|^{\frac{1}{2}} \theta_n^{\frac{1-2s}{2}} u_2\|^2
+ \tau \|e^{\tau \varphi} |\varphi''|^{\frac{1}{2}}\theta_n^{\frac{1-2s}{2}} \p_t u_2 \|^2
+ \tau \|e^{\tau \varphi} |\varphi''|^{\frac{1}{2}} 
\theta_n^{\frac{1-2s}{2}} \nabla_{S^n} u_2  \|^2\\
&\leq C (\|e^{\tau \varphi} \theta_n^{\frac{2s-1}{2}} \tilde{f}\|^2 + \tau^{2-2s}\|e^{\tau \varphi} \tilde{V} \varphi'' e^{-\alpha st} \bar{u}\|_0^2 ) + C K^4 \tau^4 \| e^{\tau \varphi} |\varphi'|^2\theta_n^{\frac{1-2s}{2}} u_1\|^2\\
&\leq C K^4 (\|e^{\tau \varphi} \theta_n^{\frac{2s-1}{2}} \tilde{f}\|^2 + \tau^{2-2s}\|e^{\tau \varphi} \varphi'' \tilde{V} e^{-\alpha st} \bar{u}\|_0^2 ) .
\end{align*}
After returning to Cartesian coordinates, this implies the claim of the theorem.
\end{proof}

\section{Proofs of Theorems \ref{thm:Landis_Diff} and \ref{thm:Landis_No_Diff}}\label{sec4}
In this section we discuss the proofs of Theorems \ref{thm:Landis_Diff} and \ref{thm:Landis_No_Diff}.

\subsection{Proof of the fractional Landis conjecture with differentiability assumptions}
\label{sec:Landis_Diff}

In this section we discuss the proof of the Landis conjecture with differentiability assumptions on the potential $q$, i.e. we present the proof of Theorem \ref{thm:Landis_Diff}. The argument for this consists of a combination of the Carleman estimate from Section \ref{sec:Carl_Diff} and the interpolation estimate from Proposition \ref{prop:interpol}.

\begin{proof}[Proof of Theorem \ref{thm:Landis_Diff}]
We begin by noticing that since for some $\alpha>1$ and some constant $C$, 
\begin{align*}
\|e^{|x|^{\alpha}/2} u\|_{L^2(\R^n)}\leq C < \infty,
\end{align*} 
Proposition \ref{prop:boundary_bulk} implies the a similar $L^{\infty}$ estimate, i.e. there exist constants $\tilde{c},\tilde{C}>0$ such that for all $R>0$
\begin{align}
\label{eq:decay_bulk_a}
|{\tu}(x)|\leq \tilde{C} e^{-\tilde{c}R^{\alpha}} \mbox{ for } x \in A_{\frac R2,3R}^+.
\end{align}
 
Next we define $w:= \eta_{R}\tu$, where $\eta_R$ is a radial cut-off function with the properties that for some constant $C>1$ independent of $R>1$
\begin{align*}
&\supp(\eta_R) \subset B_{2R}^+ \setminus B_{1}^+, \ \eta_R(x)= 1 \mbox{ for } x \in B_{R}^+ \setminus B_2^+, \\
&|\nabla \eta_R| \leq \frac{C}{R}, |\nabla^2 \eta_R| \leq \frac{C}{R^2} \mbox{ for } x \in B_{2R}^+ \setminus B_{R}^+,\\
& |\nabla \eta_R| \leq C, |\nabla^2 \eta_R| \leq C \mbox{ for } x \in B_{2}^+ \setminus B_1^+. 
\end{align*}
We note that the radial dependence of $\eta_R$ in particular entails that
\begin{align*}
\nabla \cdot x_{n+1}^{1-2s} \nabla w &= f \mbox{ in } \R^{n+1}_+,\\ 
\lim\limits_{x_{n+1}\rightarrow 0} x_{n+1}^{1-2s} \p_{n+1} w &= \eta_R \lim\limits_{x_{n+1} \rightarrow 0} x_{n+1}^{1-2s} \p_{n+1} \tilde{u} = q\eta_R u=qw \mbox{ on } \R^n \times \{0\},
\end{align*}
with $f= 2 x_{n+1}^{1-2s} \nabla \eta_R \cdot \nabla \tilde{u} + \tilde{u} \nabla \cdot x_{n+1}^{1-2s} \nabla \eta_R \in L^{2}(\R^{n+1}_+, x_{n+1}^{{2s-1}})$ being compactly supported in $\R^{n+1}_+$. 
As a consequence, the function $w$ is admissible in the Carleman estimate from Theorem \ref{thm:Carl_Diff}. Inserting it into this inequality with weight $\phi(x)=|x|^{\beta}$ for $\beta=\alpha-\epsilon>1$ {with $\epsilon \in (0,\alpha-1)$}, yields
\begin{align}
\label{eq:apply_Carl}
\begin{split}
&\tau^{3} \|e^{\tau \phi} |x|^{\frac{3\beta}{2}-1} x_{n+1}^{\frac{1-2s}{2}} w\|_{L^2(\R^{n+1}_+)}^2 + \tau \|e^{\tau \phi} |x|^{\frac{\beta}{2}}x_{n+1}^{\frac{1-2s}{2}}\nabla w\|_{L^2(\R^{n+1}_+)}^2\\
&\leq C \left( R^{-2}\| e^{\tau \phi} x_{n+1}^{\frac{1-2s}{2}} |x| \nabla \tilde{u} \|_{L^2(A_{R,2R}^+)}^2 
+ R^{-4}\| e^{\tau \phi} x_{n+1}^{\frac{1-2s}{2}} |x| \tilde{u} \|_{L^2(A_{R,2R}^+)}^2 \right.\\
& \quad  + \| e^{\tau \phi} x_{n+1}^{\frac{1-2s}{2}} |x| \nabla \tilde{u} \|_{L^2(A_{1,2}^+)}^2 
+ \| e^{\tau \phi} x_{n+1}^{\frac{1-2s}{2}} |x| \tilde{u} \|_{L^2(A_{1,2}^+)}^2\\
& \quad \left. +  \tau \|e^{\tau \phi} |x|^{\frac{\beta}{2}}|x\cdot\nabla q|^{\frac 12}w\|_{L^2(\R^n\times\{0\})}^2 +  \tau \|e^{\tau \phi} |x|^{\frac{\beta}{2}}|q|^{\frac 12}w\|_{L^2(\R^n\times\{0\})}^2 
\right).
\end{split}
\end{align}
We next discuss the bulk contributions which appear on the right hand side of the estimate (as they are lower order error contributions). Here we first focus on the contributions on the annulus $A_{R,2R}^+$: Pulling out the exponential weights and using elliptic estimates in order to bound the gradient contributions, we obtain
\begin{align*}
&R^{-2}\| e^{\tau \phi} x_{n+1}^{\frac{1-2s}{2}} |x| \nabla \tilde{u} \|_{L^2(A_{R,2R}^+)}^2 
+ R^{-4}\| e^{\tau \phi} x_{n+1}^{\frac{1-2s}{2}} |x| \tilde{u} \|_{L^2(A_{R,2R}^+)}^2 \\
& \leq R^{-1} e^{\tau \tilde{\phi}(2 R)} \| x_{n+1}^{\frac{1-2s}{2}} \nabla \tilde{u} \|_{L^2(A_{R,2R}^+)}^2 + R^{-3} e^{\tau \tilde{\phi}(2R)}\| x_{n+1}^{\frac{1-2s}{2}} \tilde{u}\|_{L^2(A_{R,2R}^+)}^2\\
& \leq C \left( R^{-3} e^{\tau \tilde{\phi}(2 R)} \| x_{n+1}^{\frac{1-2s}{2}}  \tilde{u} \|_{L^2(A_{\frac{R}{2},3R}^+)}^2+ R^{-3} e^{\tau \tilde{\phi}(2R)} \|x_{n+1}^{\frac{1-2s}{2}} \tilde u\|_{L^2(A_{\frac{R}{2}, 3R}^+)}^2 \right).
\end{align*}
Here $\tilde{\phi}:[0,\infty) \rightarrow \R$ is defined as $\tilde{\phi}(r)= \phi(x)= e^{r^{\beta}}$ with $|x|=r$ (which is well-defined since $\phi$ is a radial function). Using that by \eqref{eq:decay_bulk_a} for each $R\geq R_0\geq 1$ it holds $\|\tilde{u}\|_{L^{\infty}(A_{\frac R2,3R}^+)}\leq C e^{-c R^{\alpha}}$, we may pass to the limit $R \rightarrow \infty$ and infer that 
\begin{align*}
\lim\limits_{R \rightarrow \infty} \left( R^{-2}\| e^{\tau \phi} x_{n+1}^{\frac{1-2s}{2}} |x| \nabla \tilde{u} \|_{L^2(A_{R,2R}^+)}^2 
+ R^{-4}\| e^{\tau \phi} x_{n+1}^{\frac{1-2s}{2}} |x| \tilde{u} \|_{L^2(A_{R,2R}^+)}^2 \right) = 0.
\end{align*} 
Hence, we may pass to the limit $R \rightarrow \infty$ in \eqref{eq:apply_Carl} and deduce
\begin{align}
\label{eq:apply_Carl_limit}
\begin{split}
&\tau^{3} \|e^{\tau \phi} |x|^{\frac{3\beta}{2}-1} x_{n+1}^{\frac{1-2s}{2}} w\|_{L^2(\R^{n+1}_+)}^2 + \tau \|e^{\tau \phi} |x|^{\frac{\beta}{2}}x_{n+1}^{\frac{1-2s}{2}} \nabla w\|_{L^2(\R^{n+1}_+)}^2\\
&\leq C \left( \| e^{\tau \phi} x_{n+1}^{\frac{1-2s}{2}} |x| \nabla \tilde{u} \|_{L^2(A_{1,2}^+)}^2 
+ \| e^{\tau \phi} x_{n+1}^{\frac{1-2s}{2}} |x| \tilde{u} \|_{L^2(A_{1,2}^+)}^2 \right.\\
& \quad \left. +\tau \|e^{\tau \phi} |x|^{\frac{\beta}{2}}|x\cdot\nabla q|^{\frac 12}w\|_{L^2(\R^n\times\{0\})}^2 +  \tau \|e^{\tau \phi} |x|^{\frac{\beta}{2}}|q|^{\frac 12}w\|_{L^2(\R^n\times\{0\})}^2.  
\right).
\end{split}
\end{align}

Next, we seek to estimate the boundary contribution and show that we can absorb it into the left hand side of the estimate. For convenience and since our bulk-boundary estimates are formulated in conformal coordinates, we prove this in conformal coordinates, i.e. in the form \eqref{eq:Carl_conf_Diff}. First consider the spherical integrals. We claim that
\begin{equation}\label{eq:aux_1}
\tau |\varphi''|\||\tilde{q}|^{\frac 12}v\|_{L^2(S^{n-1})}^2
\leq C\left( \tau^{3-2s} |\varphi''||\varphi'|^2 \| \theta_n^{\frac{1-2s}{2}} v \|_{L^2(S^{n}_+)}^2 + \tau^{1-2s} |\varphi''| \| \theta_n^{\frac{1-2s}{2}} \nabla_{S^n} v \|_{L^2(S^{n}_+)}^2 \right),
\end{equation}
where $\tilde{q}(t,\theta')= e^{2st} q(e^t \theta')$ with $\theta'=(\theta_1,\cdots,\theta_n,0)$, $v(t,\theta)= e^{\tau \varphi} w(e^t \theta)$ and $\varphi(t)= \phi(e^t \theta)=e^{\beta t}$.
Indeed, we have 
\begin{align*}
%\label{eq:aux_1}
\begin{split}
|\varphi''|\||\tilde{q}|^{\frac 12}v\|_{L^2(S^{n-1})}^2
& \leq \tau |\varphi''| e^{2st} \|v\|_{L^2(S^{n-1})}^2\\
&\leq C\left( \tilde{\tau}^{2-2s} e^{2st}|\varphi''| \| \theta_n^{\frac{1-2s}{2}} v \|_{L^2(S^{n}_+)}^2 + \tilde{\tau}^{-2s} e^{2st} |\varphi''| \| \theta_n^{\frac{1-2s}{2}} \nabla_{S^n} v \|_{L^2(S^{n}_+)}^2 \right).
\end{split}
\end{align*}
Setting $e^{2st} \tilde{\tau}^{-2s} = \tau^{-2s}$, i.e. $\tilde\tau=\tau e^t$, our choice of  $\varphi(t) = e^{\beta t}$ with $\beta > 1$ and $t\geq 0$ gives
\begin{equation*}
\tilde{\tau}^{2-2s} e^{2st} |\varphi''| =\tau^{2-2s} e^{2t}|\varphi''| \leq  \tau^{2-2s} |\varphi'|^2|\varphi''|.
\end{equation*}
Multiplying \eqref{eq:aux_1} with $e^{\tau \varphi}$ and integrating in the radial variable $t\in (2,\infty)$ we hence infer
\begin{align*}
%\label{eq:aux_3}
\begin{split}
\tau\|e^{\tau \varphi} |\varphi''|^{\frac{1}{2}}|\tilde{q}|^{\frac{1}{2}} v\|_{L^2(S^{n-1} \times \R)}^2
&\leq C\left( \tau ^{3-2s} \| e^{\tau \varphi} |\varphi'| |\varphi''|^{\frac{1}{2}}  \theta_n^{\frac{1-2s}{2}} v \|_{L^2(S^{n}_+ \times \R)}^2 \right. \\
& \quad \left. + \tau^{1-2s} \| e^{\tau \varphi} |\varphi''|^{\frac{1}{2}} \theta_n^{\frac{1-2s}{2}} \nabla_{S^n} v \|_{L^2(S^{n}_+ \times \R)}^2 \right)\\
&\leq C\left( \tau ^{3-2s} \|e^{\tau \varphi} |\varphi'| |\varphi''|^{\frac{1}{2}} \theta_n^{\frac{1-2s}{2}} v \|_{L^2(S^{n}_+ \times \R)}^2 \right.\\
& \quad \left. +\tau^{1-2s}  \| e^{\tau \varphi} |\varphi''|^{\frac{1}{2}}  \theta_n^{\frac{1-2s}{2}} \nabla_{S^n} v \|_{L^2(S^{n}_+ \times \R)}^2 \right).
\end{split}
\end{align*}
A similar estimate also holds for $\tau\|e^{\tau\varphi}|\varphi'|^{\frac 12}|\partial_t\tilde q|v\|^2_{L^2(S^{n-1}\times\R)}$. Hence, choosing $\tau >1$ sufficiently large, allows us to absorb the boundary term in \eqref{eq:apply_Carl_limit} into the left hand side. Thus, we are left with
\begin{align*}
\begin{split}
&\tau^{3} \|e^{\tau \phi} |x|^{\frac{3\beta}{2}-1} x_{n+1}^{\frac{1-2s}{2}} w\|_{L^2(\R^{n+1}_+)}^2 + \tau \|e^{\tau \phi} |x|^{\frac{\beta}{2}}x_{n+1}^{\frac{1-2s}{2}} \nabla w\|_{L^2(\R^{n+1}_+)}^2\\
&\leq C \left( \| e^{\tau \phi} x_{n+1}^{\frac{1-2s}{2}} |x| \nabla \tilde{u} \|_{L^2(A_{1,2}^+)}^2 
+ \| e^{\tau \phi} x_{n+1}^{\frac{1-2s}{2}} |x| \tilde{u} \|_{L^2(A_{1,2}^+)}^2 \right).
\end{split}
\end{align*}
Thus, pulling out the exponential weight in the above estimate in particular yields
\begin{align*}
%\label{eq:apply_Carl_limit1}
\begin{split}
&\tau^{3} e^{\tau \tilde{\phi}(4)} \| x_{n+1}^{\frac{1-2s}{2}} \tilde{u}\|_{L^2(B_6^+ \setminus B_{4}^+)}^2 + \tau e^{\tau \tilde{\phi}(4)} \| x_{n+1}^{\frac{1-2s}{2}} \nabla \tilde{u}\|_{L^2(B_6^+ \setminus B_4^+ )}^2\\
&\leq Ce^{\tau \tilde{\phi}(2)} \left( \|  x_{n+1}^{\frac{1-2s}{2}} \nabla \tilde{u} \|_{L^2(A_{1,2}^+)}^2 
+ \| x_{n+1}^{\frac{1-2s}{2}}  \tilde{u} \|_{L^2(A_{1,2}^+)}^2  
\right).
\end{split}
\end{align*}
Using the monotonicity of $\tilde{\phi}$ and letting $\tau \rightarrow \infty$ leads to a contradiction unless $\tilde{u}=0$ in $B_{6}^+ \setminus B_4^+$. Then however $\tilde{u}\equiv 0$ by the weak unique continuation property, which concludes the proof for Theorem \ref{thm:Landis_Diff}.
\end{proof}

\subsection{Proof of Theorem \ref{thm:Landis_No_Diff}}
\label{sec:Landis_No_Diff}

In this section, we present the argument for Theorem \ref{thm:Landis_No_Diff}. This follows along similar lines as the proof of Theorem \ref{thm:Landis_Diff}, but now uses the Carleman estimate from Theorem \ref{thm:Carl_No_Diff} combined with the boundary-bulk interpolation result of Proposition \ref{prop:boundary_bulk}. As before, the crucial part consists in estimating the boundary contributions appropriately.

\begin{proof}[Proof of Theorem \ref{thm:Landis_No_Diff}]
As in the proof of Theorem \ref{thm:Landis_Diff} we first multiply $\tu $ with a radial cut-off function $\eta_R$ satisfying the same properties as in the previous proof. This leads to bulk contributions, which are admissible in the Carleman estimate of Theorem \ref{thm:Carl_No_Diff} with $\phi(x)=|x|^{\beta}$ and $\beta = \alpha-\epsilon> \frac{4s}{4s-1}$, i.e. $\epsilon \in (0,\alpha- \frac{4s}{4s-1})$. With a similar argument as in the proof of Theorem \ref{thm:Landis_Diff}, it is possible to pass to the limit $R\rightarrow \infty$. In conformal polar coordinates, this then leaves us with the following Carleman estimate:
\begin{align}
\label{eq:Carl_conf_no_diff}
\begin{split}
&\tau^3\|e^{\tau \varphi} |\varphi'||\varphi''|^{\frac{1}{2}} \theta_n^{\frac{1-2s}{2}} \bar{w}\|_{L^2(S^{n}_+ \times\R)}^2
+ \tau \|e^{\tau \varphi} |\varphi''|^{\frac{1}{2}}\theta_n^{\frac{1-2s}{2}} \p_t \bar{w} \|_{L^2(S^{n}_+ \times \R)}^2\\
&+ \tau \|e^{\tau \varphi} |\varphi''|^{\frac{1}{2}} 
\theta_n^{\frac{1-2s}{2}} \nabla_{S^n} \bar{w} \|_{L^2(S^{n}_+ \times \R)}^2\\
&\leq C (\|e^{\tau \varphi} \theta_n^{\frac{2s-1}{2}} \tilde{f}\|_{L^2(S^n_+\times [1,2])}^2 + \tau^{2-2s}\|e^{\tau \varphi} |\varphi''|\tilde{q} e^{-\beta st} \bar{w}\|_{L^2(S^{n-1}\times \R)}^2 ) .
\end{split}
\end{align}
Here $|\tilde{f}| \leq C \theta_n^{1-2s}(|\p_t \bar{u}| + |\nabla_{S^n} \bar{u}| + |\bar{u}|)$ and $\bar{w}(t,\theta)= \bar{u}(t,\theta)\eta_R(e^t \theta)$.
Similarly as in the proof of Theorem \ref{thm:Landis_Diff}, the crucial part consists in estimating the boundary contribution. We seek to absorb it into the left hand side of \eqref{eq:Carl_conf_no_diff}. Similarly as in the argument leading to \eqref{eq:test_ell2}, this is achieved by invoking the interpolation estimate of Proposition \ref{prop:boundary_bulk}. More precisely, using that by virtue of the choice of $\beta \geq \frac{4s}{4s-1}$ it holds that $2\beta + 4s- 2\beta s \leq \beta+ 2\beta s$, and by further setting $\tilde{\tau} = e^{-\beta t} \tau$, we deduce
\begin{align*}
&|\varphi''|^2 |\tilde{q}|^2 e^{-2\beta s t} \|\bar{w}\|_{L^2(S^{n-1})}^2
\leq C e^{(4s-2\beta s+2\beta) t} \|\bar{w}\|_{L^2(S^{n-1})}^2\\
& \leq C ( \tilde{\tau}^{2-2s}  e^{(4s-2\beta s+2 \beta) t} \|\theta_n^{\frac{1-2s}{2}} \bar{w}\|_{L^2(S^n_+)}^2 + \tilde{\tau}^{-2s}  e^{(4s-2\beta s + 2\beta) t} \|\theta_n^{\frac{1-2s}{2}} \nabla_{S^n} \bar{w}\|_{L^2(S^n_+)}^2)\\
& \leq C ( \tilde{\tau}^{2-2s}  e^{(\beta +2\beta s) t} \|\theta_n^{\frac{1-2s}{2}} \bar{w}\|_{L^2(S^n_+)}^2 + \tilde{\tau}^{-2s}  e^{(\beta +2\beta s) t} \|\theta_n^{\frac{1-2s}{2}} \nabla_{S^n} \bar{w}\|_{L^2(S^n_+)}^2)\\
& \leq C (\tau ^{2-2s}  e^{3\beta t} \|\theta_n^{\frac{1-2s}{2}} \bar{w}\|_{L^2(S^n_+)}^2 + \tau^{-2s}  e^{\beta t} \|\theta_n^{\frac{1-2s}{2}} \nabla_{S^n} \bar{w}\|_{L^2(S^n_+)}^2).
\end{align*}
Integrating this in $t\in \R$, using that $e^{\beta t}\leq C_{\beta}\min\{|\varphi'|,|\varphi''|\}$, thus results in
\begin{align*}
& \tau^{2-2s}\|e^{\tau \varphi}|\varphi''| |\tilde{q}| e^{-\beta s t} \bar{w}\|_{L^2(S^{n-1}\times \R)}^2
\leq  \tau^{1+2s}\|e^{\tau \varphi}|\varphi''|^{\frac{1}{2}} |\tilde{q}| e^{-\beta s t} \bar{w}\|_{L^2(S^{n-1}\times \R)}^2\\
& \leq C (\tau ^{4-4s} \| e^{\tau \varphi} |\varphi'||\varphi''|^{\frac{1}{2}} \theta_n^{\frac{1-2s}{2}} \bar{w}\|_{L^2(S^{n}_+\times \R)}^2 + \tau^{2-4s}   \|e^{\tau \varphi} |\varphi''|^{\frac{1}{2}}\theta_n^{\frac{1-2s}{2}} \nabla_{S^n} \bar{w}\|_{L^2(S^{n}_+ \times \R)}^2).
\end{align*}
We notice that for $s\in(\frac 14,1)$, $4-4s<3$ and $2-4s<1$. Now choosing $\tau\geq \tau_0>0$ sufficiently large, 
%or choosing the cut-off parameter properly 
implies that the boundary contribution in \eqref{eq:Carl_conf_no_diff} can be absorbed into the left hand side of this estimate. Hence, we infer
\begin{align}
\label{eq:Carl_conf_no_diff_a}
\begin{split}
&\tau^3\|e^{\tau \varphi} |\varphi'||\varphi''|^{\frac{1}{2}} \theta_n^{\frac{1-2s}{2}} \bar{w}\|_{L^2(S^{n}_+\times \R)}^2
+ \tau \|e^{\tau \varphi} |\varphi''|^{\frac{1}{2}}\theta_n^{\frac{1-2s}{2}} \p_t \bar w \||_{L^2(S^{n}_+\times \R)}^2\\
& \quad + \tau \|e^{\tau \varphi} |\varphi''|^{\frac{1}{2}} 
\theta_n^{\frac{1-2s}{2}} \nabla_{S^n} \bar{w} \||_{L^2(S^{n}_+\times \R)}^2
\leq C \|e^{\tau \varphi} \theta_n^{\frac{2s-1}{2}} \tilde{f}\|_{L^2(S^{n}_+\times [1,2])}^2 .
\end{split}
\end{align}
Pulling out the weight $e^{\tau \phi}$ in \eqref{eq:Carl_conf_no_diff_a} leads to 
\begin{equation*}
e^{\tau \varphi(4)}\tau^3\|e^{\tau \varphi} |\varphi'||\varphi''|^{\frac{1}{2}} \theta_n^{\frac{1-2s}{2}} \bar{u}\|_{L^2(S^{n}_+\times [4,6))}^2
\leq C e^{\tau \varphi(2)} \|\theta_n^{\frac{2s-1}{2}} \tilde{f}\|_{L^2(S^{n}_+\times [1,2])}^2 .
\end{equation*}
Using the monotonicity of $\varphi$ and passing to the limit $\tau \rightarrow \infty$ therefore implies that $\bar{u}= 0$ in $S^n_+ \times [4,6]$. By unique continuation this however then also gives that $\bar{u} \equiv 0$, which concludes the argument. 
\end{proof}

\section{Proof of the Quantitative Estimate of Theorem \ref{thm:quant}}
\label{sec5}

In this section, we prove the quantitative estimate from Theorem \ref{thm:quant}. To this end, we deduce bounds on the local vanishing order. Using a scaling argument as in Bourgain-Kenig \cite{BK05}, we then deduce the desired result. In order to carry out this scaling argument, we work with a slightly more general setting than in the previous sections and consider solutions to
\begin{align}
\label{eq:eq_gen}
((-\D)^{s} + q) u = 0 \mbox{ on } \R^n
\end{align}
with $q\in L^{\infty}$ but where $q$ need not necessarily be bounded by one.
The main goal of the following estimates will be the derivation of precise dependences on $\|q\|_{L^{\infty}}$.

We begin with an auxiliary result which allows us to bound  weighted gradient terms by weighted $L^2$ contributions (without boundary terms). This estimate should be thought of as an improvement of the Caccioppoli estimate from Proposition \ref{lem:Cacc} in which the boundary contribution can be eliminated due to the subcriticality of the space $L^{\infty}$ (a similar estimate holds in all subcritical spaces).

\begin{lem}
\label{lem:grad_est}
Let $s\in (0,1)$.
Let $u\in H^{s}(\R^n)$ be a solution to \eqref{eq:eq_gen}. Then, there exists a constant $C=C(n,s)>0$ such that
\begin{align*}
\|x_{n+1}^{\frac{1-2s}{2}} \nabla \tu \|_{L^2(B_1^+)}
\leq C \left( 1 + \|q\|_{L^{\infty}(\R^n)}^{{\frac{1}{2s}}} \right) \|x_{n+1}^{\frac{1-2s}{2}} \tu \|_{L^2(B_{3/2}^+)}.
\end{align*}
\end{lem}

\begin{proof}
This is a quantitative version of the proof of Proposition 2.2. in \cite{JLX11}. Following along the lines of \cite{JLX11}, we first assume that $\|q\|_{L^{\infty}} \leq \delta$ for some $\delta \in (0,\delta_0)$ sufficiently small. Then, the same argument as in the proof of Caccioppoli's inequality implies that if we test the equation for $\tilde{u}$ by $\eta^2 \tilde{u}$ we infer
\begin{align*}
\|x_{n+1}^{\frac{1-2s}{2}} \eta \nabla \tu \|_{L^2(B_1^+)}
\leq C(\| x_{n+1}^{\frac{1-2s}{2}}|\nabla \eta| \tu \|_{L^2(B_1^+)} + \|q\|_{L^\infty}^{1/2}\|\eta u\|_{L^2(B_1')} ).
\end{align*}
Here $\eta$ is a smooth, radially symmetric cut-off function, which is equal to one in $B_{2/3}^+$ and is supported in $B_1^+$. Using the smallness assumption on $\|q\|_{L^{\infty}}$ and the Poincar\'e type inequality
$\|\eta u\|_{L^2(B_1')} \leq C \|x_{n+1}^{\frac{1-2s}{2}}\nabla (\eta\tu) \|_{L^2(B_1^+)}$, we then obtain
\begin{align*}
\|x_{n+1}^{\frac{1-2s}{2}} \eta \nabla \tu \|_{L^2(B_1^+)}
\leq C(\| x_{n+1}^{\frac{1-2s}{2}}|\nabla \eta| \tu \|_{L^2(B_1^+)} + \delta_0^{1/2}\|x_{n+1}^{\frac{1-2s}{2}}\nabla (\eta\tu) \|_{L^2(B_1^+)} ).
\end{align*}
Here and below all constants $C>0$ depend on $n,s$. Choosing $\delta_0 := \frac{1}{4 C^2}$ then allows us to absorb the gradient contribution from the right hand side into the left hand side. This yields
\begin{align}
\label{eq:grad_wo_pot}
\|x_{n+1}^{\frac{1-2s}{2}} \nabla \tu \|_{L^2(B_{2/3}^+)}
\leq C\| x_{n+1}^{\frac{1-2s}{2}}\tu \|_{L^2(B_1^+)} .
\end{align}

In order to treat the general case, we consider the function $\tu_{\delta}(x):= u(\delta x + x_0)$ for $x_0 \in \R^n \times \{0\}$ arbitrary. This function still solves an equation of the type \eqref{eq:eq_gen}, but now with a potential $q_{\delta}(x):= \delta^{2s} q(\delta x + x_0)$. In particular, $\|q_{\delta}\|_{L^{\infty}} \leq \delta^{2s} \|q\|_{L^{\infty}}$. Hence, choosing $\delta>0$ such that $\delta^{2s} \|q\|_{L^{\infty}} = \delta_0$, i.e $\delta = \left(\frac{\delta_0}{\|q\|_{L^{\infty}}} \right)^{\frac{1}{2s}}$, then allows us to invoke \eqref{eq:grad_wo_pot}. 
Rescaling this and covering $B_1' \times (0,\delta/2)$ by such balls, i.e. choosing $x_j \in B_{5/4}'\times \{0\}$, $j\in\{1,\dots,N\}$, and associated balls $B_{\frac{2 \delta}{3}}^+(x_j)$, $B_{\delta}^+(x_j)$, with only finite (dimension-dependent) overlap, such that 
\begin{align*}
B_1' \times (0,\delta/2) \subset \bigcup\limits_{j=1}^{N} B_{2\delta/ 3}^+(x_j) \subset \bigcup\limits_{j=1}^{N} B_{\delta}^+(x_j)  \subset B_{5/4}' \times (0,\delta),
\end{align*}
yields for $C=C(n,s)>0$
\begin{align*}
\|x_{n+1}^{\frac{1-2s}{2}} \nabla \tu\|_{L^2(B_1' \times (0,\delta/2))}^2
&\leq \sum\limits_{j=1}^{N} \|x_{n+1}^{\frac{1-2s}{2}} \nabla \tu \|_{L^2(B_{\frac{2 \delta}{3}}^+(x_j))}^2
\leq \sum\limits_{j=1}^{N} C \left( \frac{\|q\|_{L^{\infty}}}{\delta_0} \right)^{\frac{1}{s}} \|x_{n+1}^{\frac{1-2s}{2}} \tu \|_{L^2(B_{\delta}^+(x_j))}^2\\
& \leq C \left(\frac{\|q\|_{L^\infty}}{\delta_0}\right)^{\frac{1}{s}} \|x_{n+1}^{\frac{1-2s}{2}} \tu \|_{L^2(B_{5/4}' \times (0,\delta))}^2,
\end{align*}

It remains to infer a similar estimate in $B_1' \times (\delta/2,1)$. To this end, we note that in balls $B_{r}(\bar{x}_0)$ with $\bar{x}_0=(\bar{x}_0',3r)$ and $\bar{x}_0' \in \R^n$ arbitrary, we can apply Caccioppoli's inequality without boundary contributions and with uniform constants, i.e. there exists $C=C(n,s)>0$ such that
\begin{align*}
\|x_{n+1}^{\frac{1-2s}{2}} \nabla  \tu \|_{L^2(B_r(\bar{x}_0))}
\leq \frac{C}{r} \| x_{n+1}^{\frac{1-2s}{2}}  \tu\|_{L^2(B_{2 r}(\bar{x}_0))}.
\end{align*}
Indeed, this follows from rescaling and noting that the functions $\tu_r(x):= \tu(rx',r x_{n+1})$ solve uniformly elliptic equations with uniformly bounded ellipticity constants in $B_{1}(\tilde{x}_0)$ with $\tilde{x}_0 = (\bar x'_0/r,3)$. Thus covering the domain $B_1'\times (r,1)$ with balls of the described type (which can be achieved with finite, only dimension dependent amount of overlap) and invoking a similar additive covering argument as above implies
\begin{align*}
\|x_{n+1}^{\frac{1-2s}{2}} \nabla \tu \|_{L^2(B_1'\times(\delta/2,1))}
\leq C \left( \frac{ \|q\|_{L^{\infty}}}{\delta_0} \right)^{\frac{1}{2s}}\|x_{n+1}^{\frac{1-2s}{2}} \tu \|_{L^2(B_{5/4}'\times(\delta/2,1))}.
\end{align*}
This concludes the argument.
\end{proof}

Next we approach the desired lower bound estimates. As a first step towards these, we prove a quantitative three balls inequality.

\begin{lem}
\label{lem:quantitative_3balls}
Let $s\in (1/4,1)$.
Let $u\in H^{s}(\Omega)$ be a solution to \eqref{eq:eq_gen} with $\|u\|_{L^{\infty}(\R^n)}\leq C_0<\infty$. Then there exist constants $C=C(n,s,C_0)>0$, $\alpha = \alpha(n,s) \in (0,1)$ such that for all radii $r \in (0,1)$ and all $x_0 \in \R^n \times \{0\}$ it holds
\begin{align}
\label{eq:quant_doubl}
\|x_{n+1}^{\frac{1-2s}{2}} \tu\|_{L^2(B_{2r}^+(x_0))}
\leq C e^{C(1+\|q\|_{L^{\infty}}^{\frac{2}{4s-1}})} \|x_{n+1}^{\frac{1-2s}{2}} \tu\|_{L^2(B_{r}^+(x_0))}^{\alpha}
\|x_{n+1}^{\frac{1-2s}{2}} \tu\|_{L^2(B_{4r}^+(x_0))}^{1-\alpha}.
\end{align}
\end{lem}

\begin{rmk}
The restriction $r\in (0,1)$ in the above three balls estimate is not necessary and only for convenience. As we will however only apply the result in this case and as an extension to arbitrary radii requires a slight discussion, we have added this restriction.
\end{rmk}

\begin{proof}
Without loss of generality, we prove the estimate for $x_0=0$.
To deduce the desired result, we rely on the Carleman inequality from \cite{Rue15} or equivalently the one from \cite[Proposition A.1]{GRSU18}. For a function $w \in H^1(B_5^+, x_{n+1}^{1-2s})$ with $\supp(w) \in B_4^+ \setminus B_{r_1}^+$ for some $r_1>0$ solving
\begin{align}
\label{eq:inhom}
\begin{split}
\nabla \cdot x_{n+1}^{1-2s} \nabla w & = f \mbox{ in } \R^{n+1}_+,\\
\lim\limits_{x_{n+1}\rightarrow 0} x_{n+1}^{1-2s} \p_{n+1}w & = V w \mbox{ on } \R^n \times \{0\},
\end{split}
\end{align}
with $f \in L^{2}(B_5^+, x_{n+1}^{2s-1})$ and $V \in L^{\infty}(B_5')$
it reads
\begin{align}
\label{eq:Carl_dec}
\begin{split}
&\tau^{\frac{1}{2}}\ln(r_2/r_1)^{-1} \|e^{\tau \phi} x_{n+1}^{\frac{1-2s}{2}} |x|^{-1} w\|_{L^2(B_{r_2}^+)}
+ \tau^{-s} \|e^{\tau \phi} (1+\ln^2(|x|))^{-1} |x|^{-s} w\|_{L^2(B_5')}\\
&+ \tau \|x_{n+1}^{\frac{1-2s}{2}} (1+\ln^2(|x|))^{-1} |x|^{-1} w\|_{L^2(B_5^+)}\\
&\leq C (\tau^{-1} \|e^{\tau \phi} |x| x_{n+1}^{\frac{2s-1}{2}} f\|_{L^2(B_5^+)} + \tau^{\frac{1-2s}{2}}\|e^{\tau \phi} |x|^{s} V w\|_{L^2(B_5')}).
\end{split}
\end{align}
Here $r_2 \in (2r_1, 3)$ is arbitrary and $\phi(x) = \psi(|x|)$ with
\begin{align*}
\psi(r) = -\ln(r) + \frac{1}{10}\left(\ln(r)\arctan(\ln(r)) - \frac{1}{2}\ln(1+\ln^2(r))  \right). 
\end{align*}
We apply this estimate with $w= \eta \tu$, where $\eta$ is a smooth, radial cut-off function such that
\begin{align*}
\eta(x) = 1 \mbox{ for } |x|\in (4r/5,5r/2), \ \supp(\eta) \subset A_{3r/4, 3 r}^+, \ 
|\nabla \eta| \leq \frac{C}{r}, \ |\nabla^2 \eta| \leq \frac{C}{r^2} \mbox{ for some } C>1.  
\end{align*}
The function $w=\eta \tilde{u}$ satisfies an equation of the form \eqref{eq:inhom} with $f= \tu \nabla \cdot x_{n+1}^{1-2s} \nabla \eta + x_{n+1}^{1-2s} \nabla \eta \cdot \nabla \tu$ and $V=q$.
Inserting $w$ into the Carleman estimate, we thus infer
\begin{align*}
&\tau^{\frac{1}{2}}\ln(r_2/r_1)^{-1} \|e^{\tau \phi} x_{n+1}^{\frac{1-2s}{2}} |x|^{-1} w\|_{L^2(B_{r_2}^+)}
+ \tau^{s} \|e^{\tau \phi} (1+\ln^2(|x|))^{-1} |x|^{-s} w\|_{L^2(B_5')}\\
&\leq C (\tau^{-1} r^{-1}\|e^{\tau \phi} x_{n+1}^{\frac{1-2s}{2}} \tu \|_{L^2(A_{3r/4,r}^+)} 
+ \tau^{-1} \|e^{\tau \phi} x_{n+1}^{\frac{1-2s}{2}} \nabla \tu \|_{L^2(A_{3r/4,r}^+)} \\
& \quad + \tau^{-1}r^{-1}\|e^{\tau \phi} x_{n+1}^{\frac{1-2s}{2}} \tu \|_{L^2(A_{5r/2,3r}^+)} 
+ \tau^{-1}\|e^{\tau \phi} x_{n+1}^{\frac{1-2s}{2}} \nabla \tu \|_{L^2(A_{5r/2,3r}^+)}\\
& \quad + \tau^{\frac{1-2s}{2}}\|e^{\tau \phi} |x|^{s} q w\|_{L^2(B_5')}).
\end{align*}
We first observe that if $s\in(1/4,1)$ and
\begin{align*}
\tau \geq 10 C \|q\|^{\frac{2}{4s-1}} + 1,
\end{align*}
then it is possible to absorb the boundary contribution from the right hand side into the left hand side. From now on we assume that this is the case and drop the boundary contributions. Thus, it remains to consider
\begin{align*}
&\tau^{\frac{1}{2}}\ln(r_2/r_1)^{-1} \|e^{\tau \phi} x_{n+1}^{\frac{1-2s}{2}} |x|^{-1} w\|_{L^2(B_{r_2}^+)}\\
&\leq C (\tau^{-1} r^{-1}\|e^{\tau \phi} x_{n+1}^{\frac{1-2s}{2}} \tu \|_{L^2(A_{3r/4,4r/5}^+)} 
+ \tau^{-1} \|e^{\tau \phi} x_{n+1}^{\frac{1-2s}{2}} \nabla \tu \|_{L^2(A_{3r/4,4r/5}^+)} \\
& \quad + \tau^{-1}r^{-1}\|e^{\tau \phi} x_{n+1}^{\frac{1-2s}{2}} \tu \|_{L^2(A_{5r/2,3r}^+)} 
+ \tau^{-1}\|e^{\tau \phi} x_{n+1}^{\frac{1-2s}{2}} \nabla \tu \|_{L^2(A_{5r/2,3r}^+)}).
\end{align*}
Pulling out the exponential weights and invoking Lemma \ref{lem:grad_est} in order to control the gradient contributions, we obtain
\begin{align*}
&\tau^{\frac{1}{2}}\ln(r_2/r_1)^{-1} e^{\tau\psi(r_2)}\| x_{n+1}^{\frac{1-2s}{2}} |x|^{-1} w\|_{L^2(B_{r_2}^+)}\\
&\leq C  (1+ \|q\|_{L^{\infty}}^{{\frac{1}{2s}}}) 
(\tau^{-1} r^{-1} e^{\tau \psi(3r/4)}\| x_{n+1}^{\frac{1-2s}{2}} \tu \|_{L^2(A_{{r/2,r}}^+)} \\
& \quad + \tau^{-1}r^{-1}e^{\tau \psi(5r/2)}\|x_{n+1}^{\frac{1-2s}{2}} \tu \|_{L^2(A_{{2r,4r}}^+)}) .
\end{align*}
Choosing $r_1= 3r/4$ and $r_2= 2r$ then further allows us to simplify this
\begin{align*}
&\tau^{\frac{1}{2}} r^{-1} e^{\tau\psi(2 r)}\| x_{n+1}^{\frac{1-2s}{2}}  \tu \|_{L^2(A_{4r/5, 2r}^+)}\\
&\leq C  (1+ \|q\|_{L^{\infty}}^{\frac{{1}}{2s}}) 
(\tau^{-1} r^{-1} e^{\tau \psi(3r/4)}\| x_{n+1}^{\frac{1-2s}{2}} \tu \|_{L^2(A_{{r/2},r}^+)} 
 + \tau^{-1}r^{-1}e^{\tau \psi(5r/2)}\|x_{n+1}^{\frac{1-2s}{2}} \tu \|_{L^2(A_{{2r},4r}^+)}) .
\end{align*}
Filling up holes (for which we use the monotonicity of the weight $\psi$) then yields
\begin{align*}
& \| x_{n+1}^{\frac{1-2s}{2}} \tu \|_{L^2(B_{2r}^+)}\\
&\leq C  (1+ \|q\|_{L^{\infty}}^{\frac{1}{2s}}) 
( e^{\tau \psi(3r/4)-\tau\psi(2r)}\| x_{n+1}^{\frac{1-2s}{2}} \tu \|_{L^2(B_{r}^+)} 
 + e^{\tau \psi(5r/2)-\tau\psi(2r)}\|x_{n+1}^{\frac{1-2s}{2}} \tu \|_{L^2(B_{4r}^+)}) .
\end{align*}
We optimize the right hand side in $\tau$ and choose
\begin{align*}
\tau = C(1+ \|q\|_{L^{\infty}}^{\frac{2}{4s-1}}) + \frac{1}{\psi(\frac{3r}{4})-\psi(\frac{5r}{2})} \ln \left (\frac{\| x_{n+1}^{\frac{1-2s}{2}} \tu \|_{L^2(B_{r}^+)}}{\| x_{n+1}^{\frac{1-2s}{2}} \tu \|_{L^2(B_{4r}^+)}} \right).
\end{align*}
This yields the estimate
\begin{align*}
& \| x_{n+1}^{\frac{1-2s}{2}} \tu \|_{L^2(B_{2r}^+)}\\
&\leq C  (1+ \|q\|_{L^{\infty}}^{\frac{1}{2s}}) e^{C(1+\|q\|_{L^{\infty}}^{\frac{2}{1-4s}})}
\| x_{n+1}^{\frac{1-2s}{2}} \tu \|_{L^2(B_{r}^+)} ^{\alpha}\|x_{n+1}^{\frac{1-2s}{2}} \tu \|_{L^2(B_{4r}^+)}^{1-\alpha} 
\end{align*}
where 
\begin{align*}
\alpha = \frac{\psi(5r/2)-\psi(2r)}{\psi(3r/4)-\psi(5r/2)} \in (0,1)
\end{align*}
is uniformly bounded from above and below (by the only slight convexification of the logarithmic weight).
\end{proof}

The quantitative three balls estimate from Lemma \ref{lem:quantitative_3balls} can be upgraded to a doubling inequality with a precise dependence on $\|q\|_{L^{\infty}(\R^n)}$:

\begin{prop}
\label{prop:doubling}
Let $s\in (1/4,1)$.
Let $u\in H^{s}(\Omega)$ be a solution to \eqref{eq:eq_gen} with
\begin{align*}
\|u\|_{L^2(B_1')} \geq K \mbox{ and }
 \|u\|_{L^{\infty}(\R^n)} \leq C_0.
\end{align*}
Then there exists constants $C=C(n,s,C_0)>1, \gamma=\gamma(n,s)>0$ such that for all $r\in (0,2)$ it holds
\begin{align*}
\|x_{n+1}^{\frac{1-2s}{2}} \tu\|_{L^2(B_{2r}^+)}
\leq C K^{-\gamma} e^{C (1+\|q\|_{L^{\infty}}^{\frac{2}{4s-1}})} \|x_{n+1}^{\frac{1-2s}{2}} \tu\|_{L^2(B_{r}^+)}.
\end{align*}
\end{prop}

\begin{proof}
We again apply the Carleman estimate \eqref{eq:Carl_dec} to a function of the form $w= \eta \tu$. However, we now choose $\eta$ such that for some $r\leq 1/4$
\begin{align*}
&\eta(x) = 1 \mbox{ for } |x|\in (4r/5,1), \ \supp(\eta) \subset A_{3r/4,2}^+, \\ 
&|\nabla \eta(x)| \leq \frac{C}{r}, \ |\nabla^2 \eta(x)| \leq \frac{C}{r^2}  \mbox{ for } |x| \in (3r/4,4r/5), \\  
&|\nabla \eta(x)|, |\nabla^2 \eta(x)| \leq C  \mbox{ for } |x| \in (1,2). 
\end{align*}
The Carleman estimate from above yields
\begin{align*}
&\tau^{\frac{1}{2}}\ln(r_2/r_1)^{-1} \|e^{\tau \phi} x_{n+1}^{\frac{1-2s}{2}} |x|^{-1} w\|_{L^2(B_{r_2}^+)}
+ \tau^{s} \|e^{\tau \phi} (1+\ln^2(|x|))^{-1} |x|^{-s} w\|_{L^2(B_5')}\\
&+ \tau \|e^{\tau\phi}x_{n+1}^{\frac{1-2s}{2}} (1+\ln^2(|x|))^{-1} |x|^{-1} \tu \|_{L^2(A_{4r/5,3/4}^+)}\\
&\leq C (\tau^{-1} r^{-1}\|e^{\tau \phi} x_{n+1}^{\frac{1-2s}{2}} \tu \|_{L^2(A_{3r/4, 4r/5}^+)} 
+ \tau^{-1} \|e^{\tau \phi} x_{n+1}^{\frac{1-2s}{2}} \nabla \tu \|_{L^2(A_{3r/4,4r/5}^+)} \\
& \quad + \tau^{-1}\|e^{\tau \phi} x_{n+1}^{\frac{1-2s}{2}} \tu \|_{L^2(A_{1,2}^+)} 
+ \tau^{-1}\|e^{\tau \phi} x_{n+1}^{\frac{1-2s}{2}} \nabla \tu \|_{L^2(A_{1,2}^+)}\\
& \quad + \tau^{\frac{1-2s}{2}}\|e^{\tau \phi} |x|^{s} q w\|_{L^2(B_5')}).
\end{align*}

Now we choose $\tau \geq \tau_0= C(1+\|q\|_{L^{\infty}}^{\frac{2}{4s-1}})$ so that we may absorb the boundary term from the right hand side into the left hand side. Dropping these terms again, choosing $r_1=\frac{3r}{4}$, $r_2 = 2r$, estimating the gradient contributions by Lemma \ref{lem:grad_est} and pulling out the exponential weights, leads to
\begin{align*}
%\label{eq:conse_1}
\begin{split}
&\tau^{\frac{1}{2}} e^{\tau \psi(2r)} r^{-1} \| x_{n+1}^{\frac{1-2s}{2}} \tu \|_{L^2(A_{4r/5,2r}^+)}
+ \tau e^{\tau \psi(3/4)}\|x_{n+1}^{\frac{1-2s}{2}}  \tu \|_{L^2(A_{4r/5,3/4}^+)}\\
&\leq C (1+ \|q\|_{L^{\infty}}^{\frac{1}{2s}}) (\tau^{-1} r^{-1} e^{\tau \psi(3r/4)}\| x_{n+1}^{\frac{1-2s}{2}} \tu \|_{L^2(A_{r/2, r}^+)}  + \tau^{-1} e^{\tau \psi(1)} \| x_{n+1}^{\frac{1-2s}{2}} \tu \|_{L^2(A_{9/10,21/10}^+)}) .
\end{split}
\end{align*}
Filling up holes (we again use the monotonicity of the weight $\psi$) then yields 
\begin{align}
\label{eq:conse_1}
\begin{split}
&\tau^{\frac{1}{2}} e^{\tau \psi(2r)} r^{-1} \| x_{n+1}^{\frac{1-2s}{2}} \tu \|_{L^2(B_{2r}^+)}
+ \tau e^{\tau \psi(3/4)}\|x_{n+1}^{\frac{1-2s}{2}}  \tu \|_{L^2(B_{3/4}^+)}\\
&\leq C (1+ \|q\|_{L^{\infty}}^{\frac{1}{2s}}) (\tau^{-1} r^{-1} e^{\tau \psi(3r/4)}\| x_{n+1}^{\frac{1-2s}{2}} \tu \|_{L^2(B_{r}^+)}  + \tau^{-1} e^{\tau \psi(1)} \| x_{n+1}^{\frac{1-2s}{2}} \tu \|_{L^2(B_{21/10}^+)}) .
\end{split}
\end{align}
Using the monotonicity of $\psi$ and choosing $\tau \geq \tau_0$ such that 
\begin{align*}
e^{\tau \psi(3/4)}\|x_{n+1}^{\frac{1-2s}{2}}  \tu \|_{L^2(B_{3/4}^+)} \geq 2(1+ \|q\|_{L^{\infty}}^{\frac{1}{2s}}) e^{\tau \psi(1)} \| x_{n+1}^{\frac{1-2s}{2}} \tu \|_{L^2(B_{21/10}^+)}) ,
\end{align*}
then allows us to absorb the second right hand side term into the left hand side. A possible choice for this is for instance 
\begin{align*}
\tau = \tau_0 + \frac{2(1+ \|q\|_{L^{\infty}}^{\frac{1}{2s}})}{\psi(3/4)-\psi(1)} \frac{\| x_{n+1}^{\frac{1-2s}{2}} \tu \|_{L^2(B_{21/10}^+)}}{\| x_{n+1}^{\frac{1-2s}{2}} \tu \|_{L^2(B_{3/4}^+)}}.
\end{align*}
Reinserting this into \eqref{eq:conse_1} implies 
\begin{align}
\label{eq:cons_2}
\begin{split}
\|x_{n+1}^{\frac{1-2s}{2}} \tu \|_{L^2(B_{2r}^+)}
\leq C(1+ \|q\|_{L^{\infty}}^{\frac{1}{2s}})^{\tilde C} e^{C(1+\|q\|_{L^{\infty}}^{\frac{2}{4s-1}})} \|x_{n+1}^{\frac{1-2s}{2}} \tu \|_{L^2(B_r^+)} \left(\frac{\| x_{n+1}^{\frac{1-2s}{2}} \tu \|_{L^2(B_{21/10}^+)}}{\| x_{n+1}^{\frac{1-2s}{2}} \tu \|_{L^2(B_{3/4}^+)}} \right)^{A},
\end{split}
\end{align}
where $A$ (again by the only slight logarithmic convexification of the logarithmic weight) is uniformly bounded. 

It remains to bound the quotient on the right hand side of \eqref{eq:cons_2}.
To this end, we first estimate the denominator from below and note that by the trace inequality in $H^{1}(B_2^+ , x_{n+1}^{1-2s})$ and by Lemma \ref{lem:grad_est}, we have
\begin{align*}
K \leq \|u\|_{L^2(B_1')}
\leq C (1+ \|q\|_{L^{\infty}}^{\frac{1}{2s}}) \|x_{n+1}^{\frac{1-2s}{2}}\tu\|_{L^2(B_{3/2}^+)}.
\end{align*}
By Lemma \ref{lem:quantitative_3balls} (taking $r=3/4$) and by the uniform $L^{\infty}$ boundedness of $\tu$, we obtain that for some uniform $\alpha \in (0,1)$ 
\begin{align*}
C^{-1}K(1+\|q\|_{L^{\infty}}^{\frac{1}{2s}})^{-1} \leq C C_0^{1-\alpha}e^{C(1+\|q\|_{L^{\infty}}^{\frac{2}{4s-1}})} \|x_{n+1}^{\frac{1-2s}{2}} \tu\|_{L^2(B_{3/4}^+)}^{\alpha}.
\end{align*}
For the numerator in \eqref{eq:cons_2}, we simply use the $L^{\infty}$ bound and the fact that $x_{n+1}^{\frac{1-2s}{2}}$ is $L^2$ integrable. As a consequence, for some slightly larger constant $C>1$,
\begin{align*}
\frac{\| x_{n+1}^{\frac{1-2s}{2}} \tu \|_{L^2(B_{21/10}^+)}}{\| x_{n+1}^{\frac{1-2s}{2}} \tu \|_{L^2(B_{3/4}^+)}}
\leq K^{-\frac{1}{\alpha}}C(1+ \|q\|_{L^{\infty}}^{\frac{1}{2s}})^{\tilde C}e^{C(1+\|q\|_{L^{\infty}}^{\frac{2}{4s-1}})}.
\end{align*}
Plugging this back into \eqref{eq:cons_2} and enlarging the constant $C=C(n,s,C_0)>0$ if necessary, concludes the proof of the statement.
\end{proof}

As a corollary of the doubling estimates, we derive estimates on the order of vanishing of $\tu$ and $u$.
\begin{cor}
\label{cor:order_of_van}
Let $s\in (1/4,1)$.
Let $u\in H^{s}(\Omega)$ be a solution to \eqref{eq:eq_gen} with 
\begin{align*}
\|u\|_{L^2(B_1')} \geq K >0 \mbox{ and } \|u\|_{L^{\infty}(\R^n)} \leq C_0.
\end{align*}
Then there exist constants $C_a,C_b,C_c>0$, depending on $n,s,C_0$, and $\tilde\gamma>0$, depending on $n,s$, 
such that for all $r\in (0,\frac{1}{4})$ 
\begin{equation}\label{vanishing}
(1 + \|q\|_{L^{\infty}}^{{\frac{1}{2s}}})\|x_{n+1}^{\frac{1-2s}{2}} \tu\|_{L^2(B_{r}^+)}  \geq C_a  r^{C_a (\|q\|_{L^{\infty}}^{\frac{2}{4s-1}} +1)}r^{-\tilde\gamma\log K},
\end{equation}
\begin{align*}
(1 + \|q\|_{L^{\infty}}^{{\frac{1}{2s}}})\| u\|_{L^2(B_{r}')} \geq C_b r^{C_b (\|q\|_{L^{\infty}}^{\frac{2}{4s-1}}+ 1)}r^{-\tilde\gamma\log K},
\end{align*}
and
\begin{align*}
(1 + \|q\|_{L^{\infty}}^{{\frac{1}{2s}}})\| u\|_{L^{\infty}(B_{r}')} \geq C_c r^{C_c( \|q\|_{L^{\infty}}^{\frac{2}{4s-1}}+ 1)}r^{-\tilde\gamma\log K}.
\end{align*}
\end{cor}

\begin{proof}
We first note that by the trace estimate in $H^{1}(\R^{n+1}_+, x_{n+1}^{1-2s})$ we have
\begin{align*}
K \leq \| u\|_{L^2(B_1')}
\leq C(\|x_{n+1}^{\frac{1-2s}{2}} \tu\|_{L^2(B_2^+)} + \|x_{n+1}^{\frac{1-2s}{2}} \nabla \tu\|_{L^2(B_2^+)}).
\end{align*}
By virtue of Lemma \ref{lem:grad_est} this can be further controlled by
\begin{align*}
K 
\leq C(1 + \|q\|_{L^{\infty}}^{\frac{1}{2s}}) \|x_{n+1}^{\frac{1-2s}{2}} \tu\|_{L^2(B_4^+)}.
\end{align*}
The proof of the bulk bound is now immediate from the doubling inequality. Indeed, for each $r\in (0,1/4)$ there exists $k\geq 2$ such that $r\in (2^{-k-1},2^{-k})$.
Thus,
\begin{align*}
\frac{K}{C(1 + \|q\|_{L^{\infty}}^{\frac{1}{2s}})}
&\leq \|x_{n+1}^{\frac{2s-1}{2}} \tu\|_{L^2(B_4^+)}
\leq C K^{-\gamma} e^{C (\|q\|_{L^{\infty}}^{\frac{2}{4s-1}}+1)} \|x_{n+1}^{\frac{1-2s}{2}} \tu \|_{L^2(B_2^+)}\\
&\leq C^{k+2} K^{-\gamma(k+2)} e^{ (k+2) C (\|q\|_{L^{\infty}}^{\frac{2}{4s-1}}+1)} \|x_{n+1}^{\frac{1-2s}{2}} \tu \|_{L^2(B_{2^{-k}}^+)}\\
& \leq r^{-C (\|q\|_{L^{\infty}}^{\frac{2}{4s-1}}+1) + {\tilde{\gamma} \log(K)}} \|x_{n+1}^{\frac{1-2s}{2}} \tu \|_{L^2(B_r^+)},
\end{align*}
for some $\tilde{\gamma} = \tilde{\gamma}(s,n)>0$.
This implies \eqref{vanishing}.

%Choosing $C>0$ slightly larger, then implies the claimed bulk estimate.

For the $L^2$ boundary estimate, we use the boundary bulk interpolation estimate from Proposition \ref{prop:small_prop_bound} (a). Using the upper and lower bounds for the bulk contributions and observing that the Neumann derivative again only amounts to an additional polynomial loss of $\|q\|_{L^{\infty}}$ then implies the claim. The $L^{\infty}$ estimate follows immediately from the $L^2$ bound.
\end{proof}

With the auxiliary results from Lemma \ref{lem:quantitative_3balls}-Corollary \ref{cor:order_of_van} at hand, we can address the proof of Theorem \ref{thm:quant}:

\begin{proof}[Proof of Theorem \ref{thm:quant}]
We consider the rescaled function $\tu_R(x):= \tu(R(x+\frac{x_0}{R}))$ for some $x_0 \in \R^n \times \{0\}$ with $|x_0|=R$. In particular, by rescaling, the trace inequalities in $H^{1}(\R^{n+1}_+, x_{n+1}^{1-2s})$ and by Lemma \ref{lem:grad_est}, this function satisfies 
\begin{align*}
\begin{split}
\|\tu_R\|_{L^2(B_{\frac{4}{R}}')} 
= &R^{-n/2} \|\tu\|_{L^2(B_4'(\frac{x_0}{R}))}\\
\geq& R^{-n/2}\|\tu\|_{L^2(B_2'(0))}
\geq R^{-n/2}\|u\|_{L^2(B_1')}
\geq R^{-n/2},
\end{split}
\end{align*} 
where we used that $|x_0|=R$ and $\|u\|_{L^2(B_1')}=1$. Further, we note
\begin{align*}
\|u_R\|_{L^{\infty}(\R^n\times\{0\})} \leq C_0.
\end{align*}
As a consequence, {for $R\geq 4$}, Corollary \ref{cor:order_of_van} is applicable (with $K=R^{-n/2}$) and yields
\begin{align*}
(1+\|q_R\|_{L^{\infty}}^{\frac{1}{2s}})\|u_R\|_{L^{\infty}(B_r')} \geq C_c r^{C_c (\|q_R\|_{L^{\infty}}^{\frac{2}{4s-1}}+1)}r^{n\tilde\gamma \log R}.
\end{align*}
Here $q_R(x) = R^{2s}q(R (x + \frac{x_0}{R}))$ denotes the potential in the Schrödinger equation for $u_R$. Choosing $r=\frac{1}{R}$ and using $\|q_R\|_{L^{\infty}(\R^n)} \leq R^{2s}$, we hence infer
\begin{align*}
\|u\|_{L^{\infty}(B_1'(x_0))} = \|u_R\|_{L^{\infty}(B_{1/R}')} \geq C_c (1+\|q_R\|_{L^{\infty}}^{\frac{1}{2s}})^{-1} e^{- C_c R^{\frac{4s}{4s-1}}\log R}e^{-n\gamma R\log R}.
\end{align*}
Enlarging the constant $C_c$ slightly then concludes the proof.
\end{proof}

\bibliographystyle{alpha}
\bibliography{citationsLandis_a}

\end{document}